\documentclass[12pt]{amsart}
\usepackage[left=2.00cm, right=2.00cm, top=2.00cm, bottom=2.00cm]{geometry}
\usepackage{geometry}
\usepackage{amssymb}
\usepackage{dsfont} 
	\usepackage{appendix}
\usepackage{cite}
\usepackage{color}
\usepackage{graphics}
\usepackage{graphicx}
\usepackage{mathrsfs}
\usepackage{hyperref}

	\newtheorem{theorem}{Theorem}
\newtheorem{proposition}{Proposition}

\newtheorem{Corollary}{Corollary}
\newtheorem{Lemma}{Lemma}

\theoremstyle{definition}

\newcommand{\flc}{f_{\ell,c}}

\newcommand{\C}{\mathbb{C}}

\newcommand{\N}{\mathbb{N}}

\newcommand{\R}{\mathbb{R}}

\newcommand{\Z}{\mathbb{Z}}

\newcommand{\cA}{\mathcal{A}}

\newcommand{\cC}{\mathcal{C}}

\newcommand{\cF}{\mathcal{F}}

\newcommand{\cI}{\mathcal{I}}
\newcommand{\cJ}{\mathcal{J}}

\newcommand{\cL}{\mathcal{L}}

\newcommand{\cO}{\mathcal{O}}
\newcommand{\cP}{\mathcal{P}}

\newcommand{\cU}{\mathcal{U}}

\newcommand{\teta}{\widetilde{\teta}}

\DeclareMathOperator{\Real}{Re}
\DeclareMathOperator{\meas}{m}

\DeclareMathOperator{\Ima}{Im}

\DeclareMathOperator{\dist}{dist}

\DeclareMathOperator{\HD}{HD} 
\DeclareMathOperator{\HH}{H} 

\DeclareMathOperator{\PC}{PC}
\DeclareMathOperator{\Sing}{Sing}

\title[Hyperbolic  entire transcendental maps with Baker domains]{Thermodynamic formalism for  entire transcendental maps with hyperbolic Baker domains}
                
\author{Adrián Esparza-Amador \and Irene Inoquio-Renteria}
\address{Instituto de Ciencias Físicas y Matemáticas\\
Universidad Austral de Chile, Chile}
\email{adrian.esparza@uach.cl, ireneinoquio@uach.cl}
\date{\today}
\begin{document}

\begin{abstract}
We provide a version of a thermodynamic formalism of entire transcendental maps  that exhibit Baker domains, denoted as $f_{\ell, c}: \mathbb C\to \mathbb C$ and defined by $f_{\ell, c}(z)= c-(\ell-1)\log c+ \ell z- e^z$, where  $\ell \in \mathbb N$, with $\ell \geq 2 $ and $c$ belongs to the disk $ D(\ell, 1)$ in the complex plane.
We show in particular the existence and uniqueness of conformal measures and that the  Hausdorff dimension $\HD(J_r(f))$ is the unique zero of the pressure function $t\to P(t)$, for $t>1,$ where $J_r(f)$ is the radial Julia set.
\end{abstract}


\maketitle

\section{Introduction}

For a transcendental  map $f$, the \emph{Fatou set},  denoted by $\cF(f)$ is the subset of the complex plane $\mathbb C$ where the iterates $f^n$ of $f$ form a normal family and the \emph{Julia set}, denoted by $\mathcal{J}(f)$ is the complement of the Fatou set.
A connected component $U$ of $\cF(f)$ is invariant if $f(U)\subset U.$
If $f$ is an entire transcendental map and if $U$ is an invariant component of $ \cF(f)$ then  a complete classification of all possible invariant components of $\cF(f)$ is stated in 
\cite{Ber93}.

We denote  the set of all  singularities  of $f^{-1}$ by $\Sing(f^{-1})$ which is the set of  critical and asymptotic  values of $f$. Define  the Post-singular set 
$$\cP(f):= \overline{\bigcup_{n=0}^\infty f^n(\Sing(f^{-1}))}.$$

The post-singular set play an important role in the study of the dynamics of rational maps since every periodic Fatou component is related to this set. This relation extends to the transcendental scenario for Fatou components present in the rational iteration theory,
for example, when $U$ is a B\"ottcher, Schr\"order, or Leau domain, then $U\cap \Sing (f^{-1})\neq \emptyset$. If $U$ is a Siegel disc, then $\partial U \subset \cP(f)$. Unfortunately, this relation is not clear for Fatou components that are specific for transcendental maps, that is, for \emph{Baker and wandering domains}. For example,  there are  entire maps with  Baker domains $U$, such that $U\cap \Sing(f^{-1})= \emptyset$.

On the other hand, there are condition on the post-singular set that ruled out the existence of this type of Fatou component. 
It is well known that if $f$ is an entire transcendental map in the class of Eremenko-Lyubich 
$\mathcal B:=\{f: \Sing(f^{-1}) \textrm{ is bounded}\} $, then $f$ does not have a Baker domain. Representative examples in the class $\mathcal B$ are, $f_\lambda(z)= \lambda \exp(z)$, $\lambda\in(0,1/e)$, $f_\lambda(z)= \lambda \sin(z)$, $\lambda\in (0,1)$ whose dynamics and ergodic properties  have been well studied extensively in recent times. 

Since Baker's early works \cite{Bak2, Bak3}, it has been proved that Baker domains have a very different complexity than the other Fatou components. Besides there is no direct relation between general Baker domain and post-singular set, there are different types of Baker domains according to its local dynamics near the essential singularity at infinity (see survey \cite{Rip08} for a detailed exposition on Baker domains). In the last decade several authors had make progress on the study of Baker domains, among them we refer \cite{fj23} in the entire case and \cite{bfjk15, bfjk19, em17,em21} in the general meromorphic case. 

Most of the known examples of entire transcendental maps with Baker (or wandering) domains have been considered by construction of three types: approximation theory, logarithmic lifted method, and quasi-conformal surgery. The type of transcendental maps considered in this work can be obtained through logarithmic lifted method and all of them are of \emph{hyperbolic type} according to Cowen-König classification (see \cite[Section 5]{Rip08}). 

\medskip

The thermodynamic formalism  of the entire 
transcendental dynamics presents unique challenges due to their specific  properties. The Julia set of entire transcendental maps is never compact, differing from polynomial and rational maps. Besides critical points, asymptotic values also emerge as factors to consider. For very general entire transcendental maps, the thermodynamic formalism is still an ongoing area of research and not yet fully established. As a result, it becomes necessary to focus on specific sub-classes of maps to make progress in understanding their properties.

One such subclass is that of hyperbolic entire transcendental maps.
The notion of \emph{hyperbolicity}, in this case, means that $\cP(f) $ is bounded and  $\cP(f)\cap \cJ(f)=\emptyset$, such is the case of functions in  the class $\mathcal B$.

Stallard~\cite{Sta90} introduced  a generalized  notion of hyperbolicity,  defined when $\dist(\cJ(f), \cP(f))~>0,$ with  $\dist$ representing  the Euclidean distance.  This definition extends  the concept of hyperbolicity (or expanding) of rational maps to transcendental maps and is also referred to as \emph{topologically hyperbolic}  and  \emph{$E$-hyperbolic} in the literature. 
 
If $f$ is  E-hyperbolic  and $z\in \cJ(f)$, then 
$|(f^n)'(z)|\to \infty$, as $n\to \infty$ and the  Lebesgue measure $\meas(\cJ(f))=0.$
 It is clear that a hyperbolic entire transcendental map $f$ is also $E$-hyperbolic, but the reverse is not always true. An example, is  the following entire function provided by Bergweiler in \cite{Ber95}.
\begin{equation}\label{Bergweiler}
f(z)= 2-\log 2+2z-e^z.
\end{equation}
\begin{figure}[ht]
    \centering
    \includegraphics[width=10cm]{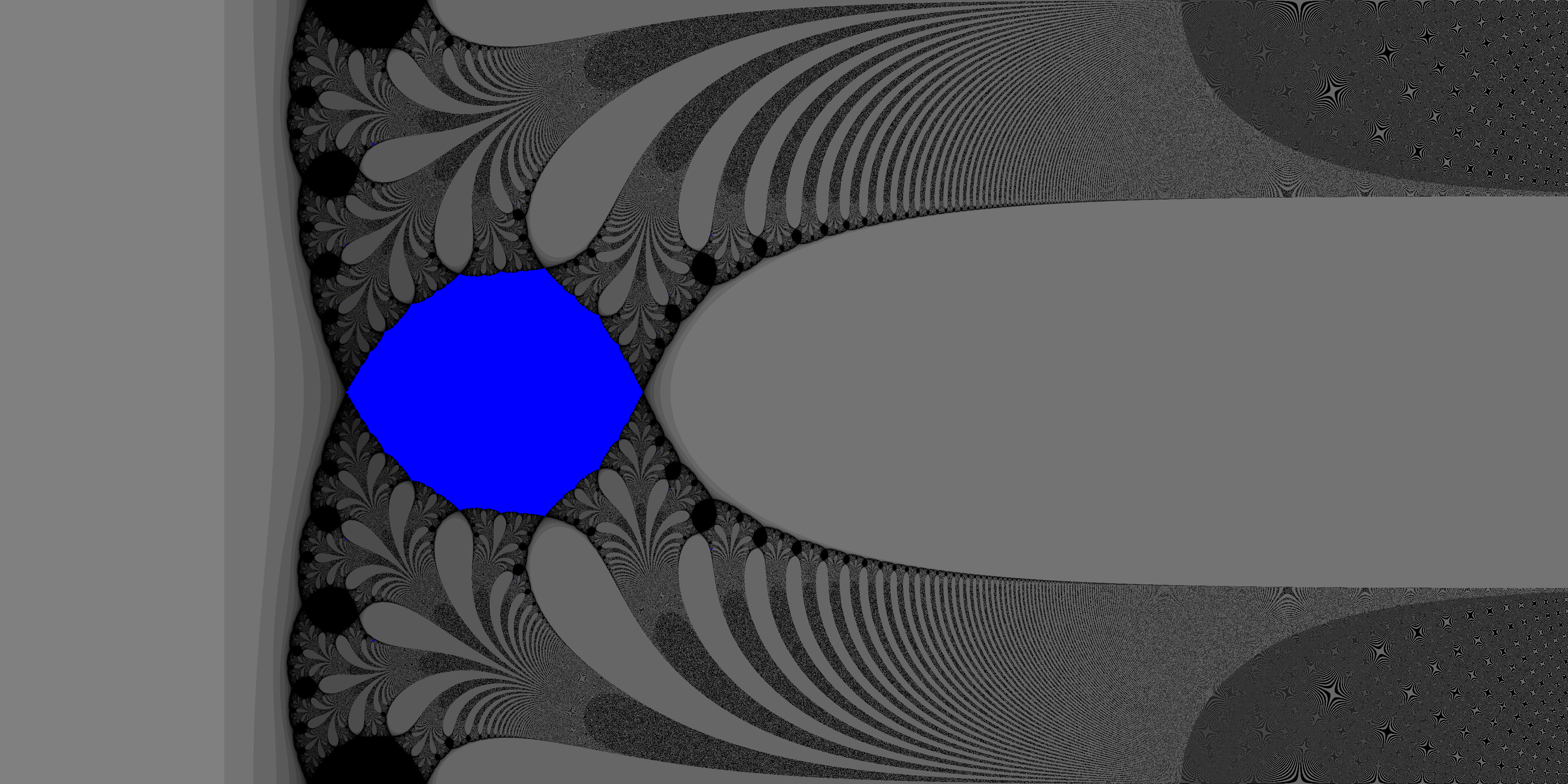}
    \caption{Dynamical plane for Bergweiler's function $f(z)=2-\log 2 +2z-e^z$. The Fatou set consists of the Baker domain (gray-tons), the super-attracting domain (blue) for $z_c=\log2$, and the (vertical) wandering domains (black). }
    \label{fig:enter-label}
\end{figure}
He proved that the function given in (\ref{Bergweiler}) has a (\emph{hyperbolic}) Baker domain
$U$ containing the left-half plane $\{z\in \mathbb C: \Real(z)<-2\}$, such that $\dist(U, \cP(f))>0. $ 
Within the proof of this result, it can be further deduced that 
$\dist(\cJ(f), \cP(f))>0$.
Additionally,  he proved that $\partial U$ is a Jordan curve in $\widehat{\C}$.

\medskip

On the one hand, Mayer and Urba\'nski in~\cite{MU10, MU20} have  made  significant  contributions to the field, by developing the thermodynamic formalism for a comprehensible  class of transcendental meromorphic  dynamically regular functions that exhibit certain derivative growth. The introduction of a suitable Riemannian metric and the application of Nevanlinna's theory led to insights such as the existence and uniqueness of Gibbs states for a class of tame potentials. However, these results might not be  directly applicable  to entire transcendental functions with Baker domains. It is 
because for most points $z$ of the Baker domain $|f(z)|$ is no greater than a multiple of $|z|$, and the iterates $f^n(z) $ do not  tend to $\infty$ rapidly.  For a comparative analysis of the growth of entire meromorphic functions that exhibit Baker domains, see~\cite[Theorem 3.1]{Rip08}.

\medskip

On the other hand,  Kotus and Urba\'nski \cite{KU05},  independently developed the  thermodynamic formalism for a class of  one-parameter family which have a (\emph{doubly parabolic}) Baker domain:
\begin{equation}\label{Fatou}
f_\lambda(z)= z+ e^{-z}+\lambda, \textrm{ with } \Real(\lambda)\ge 1,
\end{equation}
  and  potentials $-t\log |F_\lambda'|$, where $F_\lambda$  is the projection  of $f_\lambda$ to the infinite cylinder
  $\mathbb C/\sim $, and  the equivalence relation on  $\mathbb C\times \mathbb C,$ is defined by 
$z\sim w$ if and only if 
$z-w\in 2\pi i \mathbb Z.$

  This family  arises from the Fatou function $f_1(z) = z+e^{-z}+1.  $ See~\cite{Fat26}, which belongs to the Stallard's class. 


When we dynamically compare the uni-parametric family $\flc$ given as in (\ref{int_fam_fcl}) with the uni-parametric Fatou family of functions $f_\lambda$ as in (\ref{Fatou}), there are two significant differences. 
\begin{itemize}
    \item The invariant Baker domain in the family $\flc$ is of hyperbolic type, according to the Cowen-König classification, while the invariant Baker domain in the family $f_\lambda$ is of parabolic type I (or doubly parabolic). 
    \item The Fatou set in the family $f_\lambda$ is only composed of the Baker domain and its preimages. The Fatou set of the family $f_{\ell, c}$ is richer because, in addition to the Baker domain, it contains an attracting domain, and two families of wandering domains (and all of their preimages). 
\end{itemize}
Despite these differences, by working in the more friendly quotient space ${\C}/\sim$, the `control' of the post-singular set makes the study quite similar. Which is really the reason we are able to apply the techniques of Kotus and Urba\'nski to the family 
\begin{equation}\label{int_fam_fcl}
f_{\ell, c}: \mathbb C\to \mathbb C, \quad  f_{\ell, c}(z):= c-(\ell-1)\log c+ \ell z- e^z,
\end{equation}
where  $\ell \in \mathbb N$, with $\ell \geq 2 $ and $c$ belongs to the disk $ D(\ell, 1)$ in the complex plane. 

In this  work, we present  a version of the thermodynamic formalism of the entire transcendental family  $f_{\ell, c}$ which, exhibits Baker domains, thus extending the previous results of Kotus and Urba\'nski for the Fatou family.  Specifically, we show the existence and uniqueness of conformal measures and state that the  Hausdorff dimension $\HD(J_r(f))$ is the unique zero of the pressure function $t\to P(t)$, for $t>1,$ where $J_r(f)$ is the radial Julia set.


\section{Dynamics of the indexed family $\flc$}
In this section, following the ideas of Bergweiler in \cite{Ber95}, we determine the dynamical properties of each element of the family $\flc$ as defined in the introduction and the corresponding map $F_{\ell,c}$ in the quotient space $\C/\sim$. 
\subsection{The family $\flc$} 
Fix $\ell\in\N$ with $\ell\geq2$. Given $c\in D(\ell,1)$, consider the entire transcendental map $\flc:\C \to\C$ given by 
\begin{equation}\label{f_c}
    \flc(z) = c-(\ell-1)\log c + \ell z -e^z,
\end{equation}
with 
\begin{equation}\label{deri}
    \flc'(z) = \ell-e^z.
\end{equation}
\begin{theorem}\label{exBaker}
    The function $\flc$ given by (\ref{f_c}) contains a univalent invariant Baker domain $\cU$ with $\dist(\cU,\cP(\flc))>0$. Moreover, the Baker domain $\cU$ is bounded by an analytic curve. 
\end{theorem}
\begin{proof}
    From the form of $\flc$, it is clear that $\flc$ has no finite asymptotic values. Then, $\Sing(\flc^{-1})$ consists of the critical values of $\flc$, which by (\ref{deri}) are given by
    \[z_k = \log \ell +2\pi ik,\ k\in\Z.\]
    For each $z\in\C$, $k\in\Z$, we have that $\flc(z+2k\pi i)=\flc(z) + \ell(2k\pi) i$, and inductively 
    \begin{equation}\label{iter}
        \flc^n(z+2k\pi i) = \flc^n(z) + \ell^nk2\pi i,
    \end{equation}
    for $n\in\N$. On the other hand, for all $c\in D(\ell,1)$, $\flc(\log c)=\log c$ with $|\flc' (\log c)| = |\ell-c|<1$. Hence, $z_c=\log c$ is an attracting fixed point (super-attracting for $c=\ell$). Given that $c\in D(\ell,1)$ it follows that $\log \ell$ belongs to the immediate basin of attraction of $z_c$, $\cA^*(z_c)$, i.e. $\flc^n(\log \ell)\to\log c$, as $n\to\infty$. 

    Now, if $\Real(z)<-2\ell$, then 
    \begin{align*}
        \Real(\flc(z)) & = \Real(c) - (\ell-1)\log|c| + \ell\Real(z) - e^{\Real(z)}\cdot \cos (\Ima(z)) \\
            & < \Real(c) + \ell\Real(z) -e^{\Real(z)}\cos(\Ima(z)) \\
            & < \ell+1 + \ell\Real(z)-e^{\Real(z)}\cos(\Ima(z)) \\
            & \leq \ell+1 + \ell\Real(z) + e^{\Real(z)}. 
    \end{align*}
    Hence, 
    \begin{align*}
        \Real(\flc(z))     & < \ell+1 - 2\ell^2 + 1
            \leq 2\ell -2\ell^2 
             < -2\ell.
    \end{align*}
    
    This, together with the fact that $\flc(z) = c - (\ell-1)\log c + \ell z + o(1)$, as $\Real(z)\to -\infty$, implies that the half plane $\{\Real(z)<-2\ell\}$ is contained in an invariant Baker domain $\cU$ of $\flc$. 
    
    Inductively, we can prove the following, 
    \begin{equation}
        \flc^n(z) = \left(\dfrac{\ell^n-1}{\ell-1}\right)(c-(\ell-1)\log c) + \ell^nz - \sum_{j=0}^{n-1}\ell^{n-j-1}\exp(\flc^j(z)).
    \end{equation}
    Considering $h_n(z) = \flc^n(z)/\ell^n$ we may deduce that $h_n(z)$ converge as $n$ tends to infinity if $\Real(z)<-2\ell$. Also,
    \begin{equation}
        \lim_n h_n(z) = \Big(\dfrac{1}{\ell-1}\Big)(c - (\ell-1)\log c) + z - \sum_{j=0}^{n}2^{-j-1}\exp(\flc^j(z)) \sim \Big(\dfrac{1}{\ell-1}\Big)(c - (\ell-1)\log c) + z,
    \end{equation}
    as $\Real(z)\to -\infty$.

    Since $\flc$ is an entire map, it follows from \cite[Theorem 1]{Bak2} that $\cU$ is simply connected, and from \cite[Lemma 7]{Bak3} that if $K\subset \cU$ is a compact set, there exists $C>0$ such that $|f^n(z)|\leq C|f^n(z')|$ for $n$ sufficiently large, and every $z,\ z'\in K$. Hence, $\{h_n\}$ is locally and uniformly bounded in $\cU$, and so is normal in $\cU$. This way, $h_n(z)$ converge not only in $\{\Real(z)<-2\ell\}$ but for all $z\in\cU$. 

    Let $\cU_0\subset\cF(\flc)$ be the attracting component of $z_c=\log c$ containing $z_c$. For $k\in\Z$ we define $\cU_k=\{z+2\pi ik:\ z\in\cU_0\}$ and deduce from (\ref{iter}) that 
    \[
    \flc^n(z) = \log c + \ell^nk2\pi i + o(1),
    \]
    locally and uniformly in $\cU_k$ as $n\to\infty$. Which conclude that $\cU_k\subset \cF(\flc)$. Moreover, $h_n\to2\pi ki$ locally and uniformly in $\cU_k$, that is, $h_n\big|_{\cU_k}$ converge to a constant limit as $n\to\infty$. Which means that $\cU_k\cap\cU_0=\emptyset$ for every $k\in\Z\setminus\{0\}$. This also implies that 
    \[
    \text{dist}(\cP(\flc),\cU)\geq\text{dist}(\log c,\partial\cU_0)>0.
    \]
    To prove that the boundary of the Baker domain is an analytic curve in $\C$, we proceed as in \cite{Ber95}.
    
    First, note that $\flc$ is actually the \emph{logarithmic lift} of the analytic function $g:\C^*\to\C^*$, $g(z) = \dfrac{1}{c^{\ell-1}}z^\ell e^{c-z}$, which implies that $e^z\in\cJ(g_c)$ if and only if $z\in\cJ(\flc)$. 

    Let $V$ the Böttcher domain of $g_c$ containing $0$. Then $V=\exp \cU$, this way, it is enough to prove that $\partial V$ is an analytic curve. 

    We know that $\cP(g_c)=\cO^+(\ell)\cup\{0\}$. It is easy to see from this that $\text{dist}(\cP(g_c),\cJ(g_c))>0$. Let $\phi(z)=c^{\ell-1}e^{-c}z + O(z^\ell)$ the solution to the Böttcher functional equation $\phi(z^\ell)=g_c(\phi(z))$ near 0. Then $\phi$ is a conformal map from the unit disc onto $V$. Take $R$ a fixed number satisfying $0<R<1$ and define
    \[
    \gamma_n(e^{i\theta}) = \phi\left(R^{1/\ell^n e^{i\theta}}\right),
    \]
    for $n\in\N$ and $0\leq\theta\leq 2\pi$. Note that $g_c$ actually belongs to sub-class $\mathcal{B}$, so the local \emph{hyperbolicity} (see \cite{CG} and \cite{Ste}), allows proving that $\gamma_n(e^{i\theta})$ converge uniformly as $n\to\infty$, and the limit function maps the unit circle onto $\partial V$. 

\end{proof}

From the proof, given the control on the post-singular set, we obtain the following characterization of the Fatou set for each function in the family. 

\begin{Corollary}
    For $\flc$ as defined above, the Fatou set consists of the following sets: 
    \begin{itemize}
        \item A univalent invariant Baker domain $\cU_\ell$ and all of its preimages. 
        \item A simply connected (Böttcher) Schroeder domain $\cU_0$ for the (super-)attracting fixed point $z_{c}=\log c$ and all of its preimages. 
        \item Two sets of wandering domains from the $2k\pi i$-translation of the attracting domain $\cU_0$ and all of its preimages. 
    \end{itemize}
\end{Corollary}

The following result complements the structure and relation of the Baker domain and the post-singular set. 

\begin{theorem}[\cite{Ber95}, Theorem 3]
    Let $f$ be an entire transcendental function with an invariant Baker domain $U$. If $U\cap \emph{sing}(f^{-1})=\emptyset$, then there exists a sequence $\{p_n\}$ such that $p_n\in\cP(f)$, $|p_n|\to\infty$, $|p_{n+1}/p_n|\to1$, and $\emph{dist}(p_n, U)=o(|p_n|)$ as $n\to\infty$.
\end{theorem}

\subsection{Dynamics in the quotient $\C/\sim$}\label{dinamicacociente}
In this section, we proceed as in \cite{KU05}. We consider the open infinite strip 
\[
P:= \{z\in\C : 0<\Ima(z)<2\pi\}.
\]
Abusing the notation, we will think on $P$ as a subset of the infinite cylinder $Q=\C/\sim$, where the equivalence relation $\sim$ is defined on $\C\times\C$ as $w\sim z$ if and only if $w-z\in 2\pi i\Z$. The quotient space $\C/\sim$ is the infinite cylinder with the Riemann surface structure endowed by the canonical quotient mapping 
\[
\Pi : \C\to \ Q. 
\]
It follows directly that the map $\flc:\C \to \C$ respects the equivalence relation $\sim$, since 
\begin{align*}
    \flc(z+2k\pi i) & = c - (\ell-1)\log c + \ell(z+2k\pi i) - e^{z+2k\pi i} \\ 
      & = c - (\ell-1)\log c + \ell z - e^z + \ell(2k\pi i) \\
      & =\flc(z) + \ell(2k\pi i).
\end{align*}
Hence, $\flc$ induces a unique map 
\[
F_{\ell,c} : Q \to Q, 
\]
such that $F_{\ell,c}\circ \Pi = \Pi\circ \flc$. From now on, we will omit the suffix on both functions. 

Given $M\geq 0$, $D\subset\C$, and $E\subset Q$, we set 
\[
D_M=\{z : 0\leq \Real(z)\leq M\} \qquad \text{and} \qquad E_M=\{z\in Q: 0\leq \Real(z)\leq M\}. 
\]
Also, we take $D_M^c = \C\setminus D_M$, and $E_M^c=Q\setminus E_M$. 

\begin{Lemma}\label{biyectividad}
    The map $f:P\to \C$ is a bijection. 
\end{Lemma}
\begin{proof}
    We proceed by cases in the following partition of the strip $P$. Take 
    \[
    P_-:=\{z\in P : z<\Ima(z)<\pi\}, \qquad \text{and}\qquad P^-:=\{z\in P : \pi<\Ima(z)<2\pi\}.
    \]
    If $z\in P_-$, then $\sin(\Ima(z))>0$, hence 
    \begin{align*}
        \Ima(f(z)) & = \Ima(c) - (\ell-1)\arg c + \ell\Ima(z) - e^{\Real(z)}\sin(\Ima(z)) \\
           & < 2\pi + \Ima(c) - (\ell-1)\arg c \\
           & = 2\pi + \Ima(c-(\ell-1)\log c). 
    \end{align*}

    Similarly, if $z\in P^-$, then $\Ima(f(z))>2\pi + \Ima(c-(\ell-1)\log c)$. 
    Now, if $z\in P_-$, then $\Ima(f'(z))=-e^{\Real(z)}\sin(\Ima(z))<0$. 
    Since, $P_-$ is a convex subset, we concluded that 
     $f\big|_{P_-}$ 
    is injective. Analogously, if $z\in P^-$ then $\Ima(f'(z))>0$ and $f\big|_{P^-}$ is injective. Therefore, $f\big|_{P_-\cup P^-}$ is injective.

    On the other hand, if $\Ima(z)=\pi$, we have 
    \[
    \Ima(f(z)) = \Ima(c-(\ell-1)\log c) + \ell\pi i,
    \]
    and
    \[
    f(x+\pi i) = c-(\ell-1)\log c + \ell x + \ell\pi i + e^x, 
    \]
    considered as a function $f:\R\to\R$ (up to $c-(\ell-1)\log c + \ell\pi i$), we have 
    \[
    f(x+\pi i) - (c - (\ell-1)\log c) - \ell\pi i = \ell x + e^x, 
    \]
    and
    \[
    f'(x) = \ell + e^x >0
    \]
    so $f\big|_{\Ima(z)=\pi}$ is also injective. It follows that $f:P\to \C$ is injective. Finally, note that $f(x) = c-(\ell-1)\log \ell + \ell x - e^x \in f(P^-)$, and $f(x+2\pi i) = c-(\ell-1)\log c + \ell x + 4\pi i - e^x \in f(P_-)$ imply that $\partial(f(P))\subset f(\partial P)\subset f(P)$ and then $\partial f(P)=\emptyset$. This way, $f(P)=\C$ which concludes the proof. 
\end{proof}
To simplify notation, we define $\lambda:=\lambda(c) = c-(\ell-1)\log c$, and then $f(z) = f_\lambda(z)$. 
The following is the first result to show the \emph{expanding} character of the family. For this, we recall that the post-critical set of $F$, denoted by $\PC(F)$ is given by 
\begin{equation}\label{Postcritico}
\PC(F) = \overline{\{F^n(\Pi(\log \ell)):n\geq 0\}}.
\end{equation}
From the form of $\flc$ in (\ref{f_c}), it is clear that $\flc$ (and hence $F_{\ell,c}$) has no finite asymptotic values. We denote by $f_*^{-1}:\C\to P$ the (holomorphic) inverse map to the bijection $f:P\to \C$. First, we state the following weaker (than expanding) property of the map $F$. 
\begin{Lemma}\label{limsup}
    If $z\in \cJ(F)$, then $\limsup_{n\to\infty}|(F^n)'(z)|=+\infty$.
\end{Lemma}
\begin{proof}
    Proof follows verbatim from \cite[Lemma 2.2]{KU05}.
\end{proof}

That $F$ is expanding on its Julia set follows below.
\begin{proposition}\label{L:Expansion}
    There exist $L>0$ and $\kappa>1$ such that   for every $z\in\cJ(F)$ and every $n\geq1$, 
    \[
    |(F^n)'(z)|\geq L\kappa^n.
    \]
   \end{proposition}
\begin{proof}
    Proof follows verbatim from \cite[Lemma 2.6]{KU05}.
\end{proof}

\section{Topological pressure and  existence and uniqueness of conformal measures}

In this section, we will  prove the existence and uniqueness of ergodic conformal measures for $F_{\ell,c}$.
In order to simplify notation, we also  denote  $\flc$ as simply  $f$ and  $F_{\ell, c}$ as $F$. Will sometimes refer to   $\lambda$ as $c-(\ell-1)\log c$, and 
we revisit the definition of the post-critical set of $F$, given in (\ref{Postcritico}).

\medskip

Let $C_b= C_b(\cJ(F))$ be the Banach space of all bounded continuous complex values function on $\cJ(F)$.

For $t>0$, define $\mathscr L_t: C_b \rightarrow C_b$ be the Perron-Frobenius operator, given  by 
\begin{equation}\label{transfer operator}
\mathscr L_t g(z)= \sum_{x\in F^{-1}(z)}\vert F'(x) \vert^{-t} g(x).
\end{equation}
For every $n\geq 1$  its iterates are given by
$$
\mathscr L_t^n g(z)= \sum_{x\in F^{-n}(z)}\vert (F^n)'(x) \vert^{-t} g(x).$$
In particular, 
$$\mathscr L^n_t \mathds 1(z)= \sum_{x\in F^{-n}(z)}\vert (F^n)'(x) \vert^{-t}.
$$
We  will prove   in what follows that   $\vert \vert \mathscr  L_t \mathds 1\vert  \vert_\infty <\infty   $  and since  $\vert \mathscr  L_t  g(z) \vert \leq \vert \vert \mathscr  L_t \mathds 1   \vert \vert_\infty \vert \vert g\vert \vert_\infty$ we have that the operator (\ref{transfer operator}) is well-defined.

\medskip

For every $t\geq 0$ and every $z\in Q\backslash PC(F)$, define the lower and upper topological pressure, respectively by 
\begin{equation*}
    \underline{P}(t, z)=\liminf_{n\to \infty} \dfrac{1}{n}\log \mathscr L_t^n \mathds 1(z) \quad \textrm{ and }\quad
        \overline{P}(t,z)=\limsup_{n\to \infty}\frac{1}{n}\log \mathscr L_t^n \mathds 1(z).
\end{equation*}
As in \S~\ref{dinamicacociente}, we focus on the open infinite strip $P=\{z\in \mathbb C: 0< \Ima(z)< 2\pi\}$. 
Given $M\geq 0$,  we  consider the subset of the cylinder $Q$, given by 
$E_M=\{z\in Q: 0 \leq \Real(z)\leq M\}$. 
Following Lemma~\ref{biyectividad}, the map $f: P \to \mathbb C$ is bijective, so  we will denote again $f_{*}^{-1}: \mathbb C\to P$ 
as the  inverse  map of  the map $f: P\to \mathbb C.$
\begin{proposition} \label{operator}
Let $t>1$ and $z\in \cJ(F)$. Then, 
\begin{enumerate}
    \item $\underline{P}(t,z)$  and $\overline{P}(t,z)$ do not depend on the choice of $z\in \cJ(F)$. We can denote  $\underline{P}(t,z)$ and $\overline{P}(t,z)$ by $\underline{P}(t)$
 and $\overline{P}(t)$ respectively.
 \item  $\vert \vert
    \mathscr{L}_t \mathds 1
    \vert\vert_\infty: =\sup\{\mathscr{L}_t \mathds 1(z): z\in \cJ(F)\} <+\infty$.
    \item $\vert\vert \mathscr L_t^n\mathds 1(z)\vert\vert_\infty\leq \vert\vert \mathscr L_t \mathds 1(z)\vert\vert_\infty^n$ and 
       $\overline{P}(t)\leq \log \vert \vert\mathscr{L}_t \mathds 1(z)\vert \vert_\infty.  $
    \item For every $t>1$, $\overline{P}(t)<+\infty$.
    \item For $t>1,$ both functions $\underline{P}(t)$ and $\overline{P}(t)$ are convex, continuous strictly  decreasing and  $\lim_{t\to\infty}\overline{P}(t) = -\infty$.
        \item $\displaystyle\lim_{\Real{z}\to \infty} 
    \mathscr L_t \mathds 1(z)=0.$
    \end{enumerate}
\end{proposition}
\begin{proof}
(1) Since any two points in $\cJ(F)$ belong to an open simply connected set disjoint from $PC(F)$, it follows from Koebe's distortion theorem that  $\underline{P}(t,z)$ and $\overline{P}(t,z)$  are independent of $z$.

(2) For every $t>1, $ and $z\in Q\backslash PC(F)$, we have
\begin{align*}
    \mathscr L_t\mathds 1(z)= \sum_{x\in F^{-1}(z)} \vert F'(x)\vert^{-t}=\sum_{x\in F^{-1}(z)}\Big \vert 
    \ell- e^x
    \Big \vert^{-t}=
    \sum_{k=-\infty}^{+\infty}\Big \vert 
    \ell- e^{z_k}
    \Big \vert^{-t},
    \end{align*}
where  $\tilde{z}$ be the only point in $\pi^{-1}(z)\cap P$ and $z_k= f^{-1}_*(\tilde{z}+ 2\pi i k)$, with $k\in \mathbb Z$  be the only point in $P$, such that 
$f(z_k)= \tilde{z}+2\pi i k. $
Then $\ell-e^{z_k}= \ell-\ell z_k-\lambda +(\tilde{z}+ 2\pi i k)$.
So, 
$$\mathscr L_t\mathds 1(z)=
    \sum_{k=-\infty}^{+\infty} \Big\vert  \ell-\ell z_k -\lambda + (\tilde{z}+2\pi ik)\Big\vert^{-t}.
 $$
Let  $$\mathbb Z_\lambda:=\{k\in \mathbb Z: \pi \vert k-\ell \vert \geq \vert \Ima(\lambda)\vert+ 2\ell\pi \}.$$
Then, for $z\in \cJ(F)$ and $k\in \mathbb Z_k$, we have
\begin{align*}
    \Big\vert  \ell-\ell z_k -\lambda + (\tilde{z}+2\pi ik)\Big\vert\geq \Big\vert  \Ima(\ell-\ell z_k -\lambda + (\tilde{z}+2\pi ik))\Big\vert& = 
    \Big\vert  -\ell\Ima(z_k) -\Ima(\lambda) +\Ima(\tilde{z}) + 2\pi k)\Big\vert\\
    &\ge \big\vert 2\pi \vert k-\ell\vert -\vert\Ima(\lambda)\vert \big\vert \ge \pi \vert k-\ell \vert.
    \end{align*}
On the other hand,  for $z= F(z_k)$, we have $f(z_k)= \tilde{z}+ 2\pi ik$. Hence, $$\Real(z)= \Real(\tilde{z})= \Real(\lambda)+ \ell \Real(z_k)- e^{\Real(z_k)}\cos(\Ima(z_k)).$$
Then, we fix $T>0$ so large that if $\Real(z)\geq T$, then $\Real(z_k)\geq \ell$ for all  $k\in \mathbb Z.$ So, for all such $z$ and all $k\in \mathbb Z
$, we have
$$\vert \ell- e^{z_k} \vert \geq  \vert e^{z_k}\vert -\ell = e^{\Real(z_k)}-\ell\geq e^\ell-\ell > 1. $$

\medskip

Since  $\cJ(F) $  is a subset of $Q$ without critical points, and $z_k\in \cJ(F)$,  we have that the equality $\ell-e^{z_k}=0$ never holds, and therefore
$M=\inf\{\vert \ell-e^{z_k}\vert: z\in E_T, k\in \mathbb Z\backslash \mathbb Z_\lambda  \}>0.$
Then, for all $t>1$ and all $z\in \cJ(F)$ we have 
\begin{align*}
\mathscr L_t\mathds 1(z)=
    \sum_{k=-\infty}^{+\infty} \vert \ell- e^{z_k}\vert^{-t}&=
    \sum_{k\in \mathbb Z_\lambda} \Big\vert  \ell-\ell z_k -\lambda + (\tilde{z}+2\pi ik)\Big\vert^{-t}+
       \sum_{k\in \mathbb Z\backslash \mathbb Z_\lambda} \vert \ell- e^{z_{k}}\vert^{-t}\\
       & 
       \leq \sum_{z\in \mathbb Z_\lambda} (\pi \vert k-\ell\vert)^{-t} + \#(\mathbb Z\backslash \mathbb Z_\lambda) \max\{1, M\}^{-t}<+\infty. 
\end{align*}
Therefore, $\vert\vert  \mathscr  L_t\mathds 1\vert \vert_{\infty}<+\infty$.

(3) For every $n\ge 1$ and every $z\in \cJ(F),$
\begin{align*}
\mathscr L_t^n\mathds 1(z)= \sum_{x\in F^{-n}(z)} \vert  (F^n)'(x)\vert^{-t}&=  \sum_{y\in F^{-(n-1)}(z)} \vert  (F^{n-1})'(y)\vert^{-t}
\sum_{x\in F^{-1}(y)} \vert  F'(x)\vert^{-t}\\
&\le \vert\vert  \mathscr L_t \mathds 1\vert \vert_{\infty} \mathscr L_t^{n-1} \mathds 1(y).
\end{align*}
So, by a direct inductive argument, we have  that $|\mathscr L_t^n\mathds 1(z)| \leq \| \mathscr L_t\mathds 1 \|^n_\infty $ and consequently  for all $t>1$,   $\overline{ P}(t)=\overline{ P}(t,z)\le \log \| \mathscr L_t \mathds 1 \|_{\infty}$.

Item (4) follows from Items (2) and (3).
Item (5) follows immediately from H\"older's inequality. Thus, thanks to the convexity,  the function $t\to \overline{P}(t)$, with $t>1$ is continuous. The facts that $P(t)$, $t>1$ is strictly decreasing and that $\lim_{t\to +\infty}P(t)= -\infty$, follows from Proposition~\ref{L:Expansion}.

Finally, to prove Item (6), we consider $z\in \cJ(F)$ and $\tilde{z}\in \Pi^{-1}(z)\cap P.$ By definition of $f_{\ell,c}, $  it follows that,  $\displaystyle\lim_{\Real z \to +\infty} \Real z_k = +\infty, $ uniformly with respect to $k\in \mathbb Z. $
Hence, 
$\displaystyle\lim_{\Real z\to +\infty} |\Real (\tilde{z}-\ell z_k)|= +\infty$
uniformly with respect to $k\in \mathbb Z. $ Then, 
\begin{align*}
    |\ell- \ell z_k-\lambda+ (\tilde{z}+ 2\pi i k)|&\geq  \frac{1}{2} \Big(
    |\Real\left(\ell- \ell z_k-\lambda+ (\tilde{z}+ 2\pi i k)\right)|+
    |\Ima\left(
    \ell- \ell z_k-\lambda+ (\tilde{z}+ 2\pi i k)
    \right)|\Big)\\
    & \geq  \frac{1}{2} \left(\Real(\tilde{z}-\ell z_k)\right)+ (\ell +\Real(\lambda))+\frac{1}{2} \left(- 2\ell \pi- \Ima(\lambda)+ 2 \pi \vert k\vert   \right)   \\
    & \geq M+ \frac{\pi}{2}\vert k\vert.
\end{align*}
Therefore, 
$$\mathscr L_t(\mathds 1)(z)\leq \sum_{k\in \mathbb Z} (M+ \frac{\pi}{2}\vert k\vert )^{-t}.$$ Then, letting $M\nearrow +\infty$, we get the desired result.
\end{proof}

\subsection{Existence of conformal measure}
Let $t, \alpha \in \mathbb R. $ We say that a measure $\nu$ on $\cJ(F)$ is  $(t, \alpha)$-\emph{conformal}  if for every Borel set $A\subset \cJ(F)$ such that $F\vert_{A}$ is injective, then 
$$\nu(F(A))= \int_A \alpha \vert  F'\vert^{t}\,  d\nu. $$

Fix $n\geq 1$. Consider the set 
$$K_n = \bigcap_{j\ge 0} F^{-j}(E_n), $$
 where $E_n=\{z\in Q: 0\le \Real(z)\le n\}$. 
 Since $F: Q\to Q$ is continuous and $K_n$ is $F$-forward invariant  compact subset of $Q$. Following \cite[Lemma 5.3]{DU91a} and \cite{PU}, we have that for every $t>0$ and  for every $n\geq 1$ there exist  a Borel probability measure $m_n$ supported on $K_n$ and a non-decreasing sequence $\{P_n(t)\}_{n=1}^\infty$ satisfying the semi-conformality property:
 $$m_n(F(A))\ge e^{P_n(t)}\int_A\vert F'\vert^{t} \, dm_n, $$
  for every Borel set  $A\subset E_n$ such that $F\vert_A$ is injective.
  If, in addition, $A\cap \partial E_n= \emptyset$,  then 
 $$m_n(F(A))= e^{P_n(t)}\int_A\vert F'\vert^{t} \, dm_n.$$
 
\medskip


\begin{proposition}\label{tight}
The sequence of measures $\{m_n\}_{n=1}^\infty$ is tight, that is,  for every $\epsilon >0$, there exists  a compact set $C$ of $\cJ(F)$ such that $m_n(\cJ(F)\backslash C)<\epsilon$, for all $n.$ 
\end{proposition}

\begin{proof}
In this proof, for the sake of simplicity, we will use 
 $E_n^c$ to actually denote 
$E_n^c\cap \{z\in Q: \Real(z)>0\}$. 
Moreover, we will keep notation $\lambda= c-(\ell-1)\log c$.

\medskip

Let 
\[E_1(M):= \{z\in \cJ(F):  \Real(F(z))\geq M\};\]
\[E_2(M):= \{z\in \cJ(F): \Real(z)\geq M\ \&\ \Real(F(z))< M\};\]
We will estimate separately the measure  $m_n$ of  the sets $E_1(M)$
 and $E_2(M)$, which cover $E_M^c.$
 
Fix $\epsilon >0$ and $M>0$, by using the definition of semi-conformality of $m_n$, we have
\begin{align}\label{ine_conf}
m_n\Big(E_1(M)\Big) 
\leq & 
m_n\Big( F^{-1}\Big(\{z\in Q: \Real(z)\geq M\} \Big)\Big)\nonumber\\ = & m_n\Big(\bigcup_{k\in \mathbb Z} F_k^{-1}\Big(\{z\in Q: \Real(z) \geq  M\}\Big) \Big)\nonumber\\
=&\sum_{k\in \mathbb Z} m_n\Big(F_k^{-1}(E_{M}^c )\Big) \\
\leq &\sum_{k\in \mathbb Z} e^{-P_n(t)} \dfrac{1}{\displaystyle\inf_{w\in F^{-1}(E_M^c)} \vert (F_k)'(w)\vert^t } m_n(E_M^c) \nonumber
\\
& \leq  e^{-P_n(t)} \sum_{k\in \mathbb Z}\sup_{z\in E_M^c} \vert (F_k^{-1})'(z)\vert^t.\nonumber
\end{align}
Here, for every $k \geq 1,$ $F_k^{-1}: \{z\in Q: \Real(z)\geq M\}\to Q $ is the inverse branch  $F_k^{-1}(z)= f_*^{-1}(z+ 2\pi k i)$.

\medskip
Recall that, for every $z\in Q, $ 
$z_k = F_k^{-1}(z)= f_*^{-1}(z+2k \pi i).$
then,  for  
 $z\in E_M^c$, we have
$$
\vert (F_k^{-1})'(z) \vert =\dfrac{1}{\vert \ell-e^{z_k}\vert}= \dfrac{1}{\vert \ell+z+ 2\pi k i-\lambda - \ell f_{*}^{-1}(z+2\pi ki)\vert }.$$
Note that  from the equality
$
f(z_k)=\lambda+ \ell z_k-e^{z_k}= z+ 2\pi k i, 
$
and 
$$
\ell \Real(z_k)= \Real(z)+ e^{\Real(z_k)} \cos(\Ima(z_k))- \Real(\lambda), 
$$ we have,  for $M>0$ large enough,  $\Real(z_k)$ becomes as large as one desired uniformly with respect to $k\in \mathbb Z.$

Moreover, since $0\le \Ima(z_k)\leq 2\pi,$ then 
$\vert z+ 2\pi ki \vert =
 \vert \lambda+\ell z_k-e^{z_k}\vert  \geq \dfrac{1}{2}\vert e^{z_k}\vert= \dfrac{1}{2}e^{\Real(z_k)}.$
Then 
\[\Real (z_k)\leq \log 2+ \log \vert z+2\pi ki\vert.\]
Thus, 
\begin{equation}\label{eq: Rea(z)-Rea(zk)}
\vert \Real (z)-\ell\Real(z_k)\vert\geq \Real(z)-\ell\Real(z_k)
= \Real(z)- \ell\log 2- \ell\log \vert z+2\pi k i \vert.
\end{equation}
Consider 
$$
\hat k=\max\Big\{k\geq 0: 
\log \big\vert z+ 2\pi l i\big\vert\leq \frac{1}{3}\Real(z), \forall l\in \mathbb Z,  \,  \vert l \vert \leq k
\Big\}.
$$
Then 
$\log \big\vert z+ 2\pi i (\hat k+1) \big\vert\geq  \frac{1}{3} \Real(z)$ or equivalently,  
\begin{equation}\label{eq:hat k}
\big\vert z+ 2\pi i (\hat k+1) \big\vert\geq  e^{\frac{1}{3} \Real(z)}.
\end{equation}
Since
$\vert z+ 2\pi i (\hat k+1)\vert\leq \vert z\vert + 2\pi (\hat k+1)\leq\frac{1}{2} e^{\frac{1}{3}\Real(z)}+ 2\pi (\hat k+1), $
then  from (\ref{eq:hat k}), we have 
$e^{\frac{1}{3}\Real(z)}\leq \frac{1}{2} e^{\frac{1}{3}\Real(z)}+ 2\pi (\hat k+1).$
Then, assuming that $M$ is large enough, we get
$$4\pi \hat k\geq 2 \pi(\hat k+1)\geq \frac{1}{2}e^{\frac{1}{3}\Real(z)},$$
then, $\hat k\geq \frac{1}{8\pi}e^{\frac{1}{3}\Real(z)}\geq e^{\frac{1}{4}\Real(z)}$.
Thus, 
$$e^{\frac{M}{4}}\leq e^{\frac{1}{4}\Real(z)}\leq \hat k.$$
So, for $M>0$ large enough, and $\vert k\vert \leq \hat k$ and using (\ref{eq: Rea(z)-Rea(zk)}), we have
\begin{align*}
    \vert \Real(z)- \ell\Real(z_k)\vert\geq & \Real(z)- \ell\log 2-\frac{2}{3} \Real(z) 
    = \frac{1}{3}\Real(z)- \ell\log 2
    \geq \frac{1}{4}\Real (z)\geq \frac{M}{4}.
\end{align*}
Therefore 
\begin{align*}
    &\vert (z+2k\pi i)-\ell z_k-(\lambda-\ell) \vert \\
    &\geq  \frac{1}{2}\vert \Real((z+2k\pi i)-\ell z_k-(\lambda-\ell))\vert+ \frac{1}{2}\vert \Ima((z+2k\pi i)-\ell z_k-(\lambda-\ell))\vert \\
    & \geq \frac{1}{2}|\Real(z)-\ell\Real(z_k)-\Real(\lambda-1)|+\frac{1}{2} |\Ima(z)+2k\pi-\ell\Ima(z_k)-\Ima(\lambda)| \\
    & \geq \frac{1}{2}|\Real(z)-\ell\Real(z_k)| - \frac{1}{2}|\Real(\lambda-\ell)|+\frac{1}{2}2\pi|k|-2\ell\pi-\Ima(\lambda)) \\
    &\geq \frac{M}{8}-\frac{1}{2}|\Real(\lambda-\ell)| + \pi|k| \geq \frac{M}{9}+\pi|k|.
\end{align*}
If $\vert k\vert\geq \hat k, $
then,  the following inequality holds if only $M$ is sufficiently large.
\begin{align*}
    \vert (z+2k\pi i)-\ell z_k-(\lambda-\ell)\vert\geq 
\vert \Ima&((z+2k\pi i)-\ell z_k-(\lambda-\ell))\vert\\
\geq 2\pi& \vert k\vert- 2\ell\pi - \Ima(\lambda)\\
> \pi &\vert k \vert.
\end{align*}

Thus, we get
\begin{align*}
    \sum_{k\in \mathbb Z} e^{-P_n(t)}\sup_{z\in E_M^c}\big\{\vert (F_k^{-1})'(z) \vert\big\}^{t} &=
    e^{-P_n(t)}\sum_{\vert k\vert \leq \hat k} \sup_{z\in E_M^c}\big\{\vert (F_k^{-1})'(z) \vert\big\}^{t}
    +
    e^{-P_n(t)}\sum_{\vert k\vert > \hat k} \sup_{z\in E_M^c}\big\{\vert (F_k^{-1})'(z) \vert\big\}^{t}\\
    &\leq 
    e^{-P_n(t)}\sum_{\vert k\vert \leq \hat k} \Big(\frac{M}{9}+ \pi \vert k\vert\Big)^{-t}
    +
    e^{-P_n(t)}\sum_{\vert k\vert > \hat k} (\pi\vert k \vert)^{-t}
    \\
    &\leq 
    2e^{-P_n(t)}\sum_{k=0}^{\infty} \Big(\frac{M}{9}+ \pi \vert k\vert\Big)^{-t}
    +
    \frac{2}{\pi^t}e^{-P_n(t)}\sum_{k > \hat k} k^{-t}  \\  
&\leq
    A_t^{(1)} e^{-P_n(t)} M^{1-t} +A_t^{(2)} e^{-P_n(t)} \hat{k}^{(1-t)}\\
    & \leq A_t e^{-P_n(t)} \max\Big\{ M^{1-t} , e^{\frac{M}{4}(1-t)}\Big\}\\
    &= A_t e^{-P_n(t)} M^{1-t}.
        \end{align*}
        Where, $ A_t^{(1)}, A_t^{(2)}, A_t$ are  constants depending only on $t>1$  and 
        this always holds for $ M $ sufficiently large.

 Since $\sup \big\{\vert F'(z) \vert: z\in K_n\big\} <\infty$  then,  for $n$ large enough we have 
        $P_{n}(t)>-\infty $. Since $\{P_n(t)\}_{n=0}^{\infty}$ is non-decreasing, we  put $\gamma(t)= \sup_{n} \{-P_n(t)\}< \overline{P(t)}<\infty $.
       Then, for $M$ sufficiently large,    
\begin{equation}\label{primerConjunto}
m_n\Big(E_1(M) \Big)\leq A_t e^{\gamma(t)}M^{(1-t)}.
\end{equation}

We now estimate $m_n (E_2)$. Note first, that for  $z\in E_2(M), $  $\Real(z)\geq M$ and $ \Real(f(z))<M$. Then, 
\[
|f(z)| = |\lambda + \ell z-e^z| \geq |e^z|- |\ell z- \lambda| \geq |e^z|= e^{\Real(z)} \geq e^M.
\]
Hence, 
\[
e^{2M}\leq |f(z)|^2 = |\Real (f(z))|^2+  |\Ima(f(z))|^2\leq M^2+ |\Ima(f(z))|^2.
\]
Thus, 
\[ |\Ima(f(z))|\geq \sqrt{e^{2M}- M^2} \geq e^{M}/2 .\]
For every $k\geq e^M/2$, and for $M$ large enough, 
\[
|\ell+z+2\pi ki -\lambda- \ell f_*^{-1}(z+ 2\pi k i)|\geq 2\pi k- 2(\ell+1) \pi- |\lambda-\ell|\geq k.
\] So, 
$|(F_k^{-1})'(z)|^t\leq  k^{-t}.$

Then  for $M$ large enough, 
\begin{equation}\label{segundoconjunto}
m_n (E_2(M))\leq e^{\gamma(t)}\sum_{k\geq e^M/2} k^{-t}\preceq e^{\gamma(t)} e^{M(1-t)}.
\end{equation}

Therefore combining~(\ref{primerConjunto}) and (\ref{segundoconjunto}), we obtain that the sequence 
$\{m_n\}_{n\geq 1}$ is tight. 
        
\end{proof}

  Recall that for $n$ sufficiently large,  $P_n(t)\leq \overline{P}(t)$ and the sequence $\{P_n(t)\}$ is non-decreasing. We denote $\log (\alpha_t):= \lim_{n\to \infty} P_n(t)$. 
The existence of  a weak limit  of $\{m_n\}_n$  derives from Prochorov's Theorem's following the same argument of reasoning  found in \cite{KU05} based on the work  presented  in  \cite{DU91a}. By  considering  a subsequence $\{n_k\}_{k=1}^\infty$, we  can get the  desired conformal measure  $m_t$ (with certain modification on the boundary sets), as  a weak limit of  sequence $\{m_{n_k}\}_k$.
Concretely,  the following result holds.
\begin{theorem}\label{lim_meas}
For every $t>1$, there exist $\alpha_t>0$ and  a $(t, \alpha_t)$-conformal measure $m_t$ for the map $F: \cJ(F)\to \cJ(F),$ and $m_t(\cJ(F))=1$.
\end{theorem}

\medskip

Consider the transfer operator
$$\mathscr L_t g(z)=\sum_{x\in F^{-1}(z)}\vert F'(x)\vert^{-t} g(y), \quad \quad \forall \; g\in  C_b(\mathcal \cJ(F)).$$
We know from Proposition~\ref{operator} Item~(2), $\mathscr L_f$ is a well-defined linear operator.

Let $\mathscr L_t^*: C_b (\cJ(F))^*\to C_b (\cJ(F))^*$   denote the dual operator of $\mathscr L_t$ defined by
	$$
	\int \phi \, d(\mathscr L_t^*\mu)=\int \mathscr L_t \phi \,d\mu, \quad \forall \, \phi \in C_b(\cJ(F)).
	$$
Define $\widehat{\mathscr L}_t = \alpha_t^{-1} \mathscr L_t$ and denote its dual operator by $\widehat{\mathscr L}_t^*.$

\begin{proposition}
Let $t>1$,  the following properties hold
\begin{itemize}
\item[(i)] $\mathscr L_t^*m_t= \alpha m_t, $ or equivalently  $\widehat{\mathscr L}_t^*m_t= m_t$.
\item[(ii)]  
$\sup_{n\geq 0} \{\vert \vert  
\widehat{\mathscr L}_t^n (\mathds 1)
\vert \vert_\infty  \}<\infty$.
\item[(iii)] For every $\epsilon>0$ there exists $M>0$ such  that 
$$
\inf_{n\geq 0}\sup\{
\widehat{\mathscr L}_t^n (\mathds 1)(z): \vert \Real z\vert\leq M 
\}\geq 1-\epsilon.
$$
\item[(iv)]  There exists  $M_0$ such that  for every  $M\leq M_0$,
$$
\inf_{n\geq 0}\inf\{
\widehat{\mathscr L}_t^n (\mathds 1)(z): \vert \Real z\vert\leq M 
\}\geq\frac{1}{4K_m^t} \Big(\max\{ K_M, K_{M_0}\}\Big)^{-1}.
$$
\item[(v)] $P(t)= \overline{P}(t)= \underline{P}(t)= \log \alpha_t.$
\end{itemize}

\end{proposition}

\begin{proof}
(i) We refer to \cite{DU91b} for the standard proof. 

(ii) 

First, observe as in \cite{KU05},  that fixed any two points $z, w \in Q\backslash PC(F)$, then there exists  a shorter smooth arc   $\gamma_{z,w}, $ joining $z$ and $w$ in $ Q\backslash B(PC(F), 2 \delta)$, where
$\delta = \frac{1}{2} \min \Big\{\frac{1}{2}, \dist (J(F), PC(F)
\Big\}.$ There exists a number $l_M\geq 1$ such that each such arc $\gamma_{z,w}$ can be covered by a chain of most balls $\ell_M$ of radius $\delta$ centered at point of $\gamma_{z,w}$. Assume that  the union of thes balls $U_{z,w}$ is a simply connected set.  Then  from Koebe's distortion theorem follows that there exists $K_M\geq 1$ such that  $F^{-n}: U_{z,w}\to Q$ is a holomorphic branch of $F^{-n}$. Then, 
\[
\frac{|(F^{-n}_{*})'(z)|}{|(F^{-n}_{*})'(w)|}\leq K_M.
\]
Consequently 
\[
K_M^{-t}\leq \frac{\widehat{\mathscr L}_t^n(\mathds 1)(z)}{\widehat{\mathscr L}_t^n(\mathds 1)(w)}\leq K_M^t.
\]
In view of Proposition~\ref{operator}, Item~(5), there exists $M\geq 0$ so large in absolute value that for every $w\in E_M^c$, 
$$
\alpha_t^{-1} \mathscr L_t (\mathds 1) (w)\leq 1.
$$
Inductively, we have for every $n\geq 0$, 
$$
\vert \vert \widehat{\mathscr L}_t^n(\mathds 1)\vert \vert_{\infty} \leq \frac{K_M^{t}}{m_t(E_M)}.
$$
For  $n=0$ this estimate is immediate. Suppose  it holds for some $n\geq 0$,  and let $z^{n+1}\in Q$, such that 
$$
\widehat{\mathscr L}_t^{n+1}(\mathds 1)(z^{n+1})= \vert \vert  \widehat{\mathscr L}_t^n(\mathds 1)\vert \vert_{\infty}.
$$
If $z^{n+1}\in E_M$, then by Distortion Koebe Lemma we have 
\begin{align*}
    1  = \int  \widehat{\mathscr L}_t^{n+1}(\mathds 1) dm &\geq \int_{E_M} \widehat{\mathscr L}_t^{n+1}(\mathds 1) dm \\ & \geq 
    \inf_{E_M} \widehat{\mathscr L}_t^{n+1}(\mathds 1) m(E_M)
    \geq  K_M^{-t}\sup_{E_M} 
    \widehat{\mathscr L}_t^{n+1}(\mathds 1)  m(E_M) \\ & =
    K_M^{-t} \vert \vert \widehat{\mathscr L}_t^{n+1}(\mathds 1)\vert \vert_{\infty} m(E_M).
\end{align*}
and consequently,  
$\vert \vert
\widehat{\mathscr L}_t^{n+1}(\mathds 1) 
\vert\vert_\infty\leq  K_M^t (m(E_M))^{-1}.$

\medskip
If $z^{n+1}\notin E_M$,  then 
\begin{align*}
\vert\vert \widehat{\mathscr L}_t^{n+1}(\mathds 1) \vert\vert_\infty & = \widehat{\mathscr L}_t^{n+1}(\mathds 1)(z^{n+1})=\alpha_t^{-1}\sum_{y\in F^{-1}(z^{n+1})}
\vert F'(y)\vert^{-t} \widehat{\mathscr L}_t^{n}(\mathds 1)(y)\\
&\leq \alpha_t^{-1} \sum_{y\in F^{-1}(z^{n+1})} \vert F'(y) \vert^{-t} \vert\vert   \widehat{\mathscr L}_t^{n}(\mathds 1) \vert \vert_\infty
\leq 
\widehat{\mathscr L}_t^{n}(\mathds 1)(z^{n+1}) \frac{K_M^t}{m_t(E_M)}\leq \frac{K_M^t}{m_t(E_M)}.
\end{align*}
(iii) Let $\Theta= \sup_{n\geq 1} \vert \vert  \widehat{\mathscr L}_t^{n}(\mathds 1) \vert \vert_\infty<\infty$. 
 Let $M\geq 0, $ be so large that  $m_t (E_M^c)\leq \epsilon/ (4\theta). $
 Suppose for the contrary  that exists  $\epsilon>0$,  such that  for every $M>0,$
 $$\inf_{n\geq 0}\sup\{\widehat{\mathscr L}_t^{n}(\mathds 1)(z): \vert\Real z \vert\leq M \}< 1-\epsilon. $$
 Then, 
\begin{align*}
1 & =\int \widehat{\mathscr L}_t^{n}(\mathds 1) d m_t = \int_{E_M}\widehat{\mathscr L}_t^{n}(\mathds 1) d m_t+ \int_{E_M^c} \widehat{\mathscr L}_t^{n}(\mathds 1) d m_t \\ &\leq 
(1-\epsilon) m_t(E_M)+ \Theta m_t (E_M^c)\leq (1-\epsilon)+ \Theta\frac{\epsilon}{4\Theta} = 1-\frac{\epsilon}{4}.
    \end{align*}
    (iv) Follows immediately from (iii) and the Koebe Distortion theorem.
    \medskip

    Finally, to prove (v), let $\Theta= \sup_{n\geq 1} \vert \vert \widehat{\mathscr{L}}_t^n(\mathds 1)\vert \vert_\infty <\infty$. From Part (ii) 
        we have that for every $z\in P, $
        $$
    \mathscr{L}_t^n(\mathds 1)(z)\leq \Theta\alpha_t^n.
        $$
    Hence, 
    $$\overline{P}(t)= \limsup_{n\to \infty}\frac{1}{n}\log \mathscr{L}_t^n(\mathds 1)\leq \log \alpha_t.$$
    Moreover, in view of Part (iv), 
$$\mathscr{L}_t^n(\mathds 1)(x_0)\geq \frac{1}{4 K_{M_0}}\alpha_t^n$$ and therefore 
$$\underline{P}(t)=  \liminf_{n\to \infty}\frac{1}{n}\log \mathscr{L}_t^n(\mathds 1)\geq \log \alpha_t.$$
\end{proof}

Now, we aim to prove ergodic and uniqueness property of the limit measure obtained in the Theorem \ref{lim_meas}.

Recall that, we consider the escaping subset of the Julia set of $F$ defined by 
\[
I_\infty(F) = \{z\in\cJ(F):\lim F^n(z) = +\infty\}.
\]
Note that the above set is well-defined since $F$ acts on the quotient space $Q$ and $f$ has a Baker domain containing a left half-plane. So, the analogous set for $f$ is given by 
\[
I_\infty(f)=\{z\in\cJ(f):\lim_{n\to\infty}\Real(f^n(z))=+\infty\}. 
\]
It is clear that 
\[
I_\infty(f) =\Pi^{-1}(I_\infty(F)),
\]
where $\Pi:\C\to Q$ is the quotient map. 
We denote by $\cJ_r$ the corresponding complement of the escaping, that is, 
\[
\cJ_r(F)=\cJ(F)\backslash I_\infty(F),\quad\text{and}\qquad \cJ_r(f)=\cJ(f)\backslash I_\infty(f). 
\]
We first prove a general result for conformal measures on the dynamics of $F$. 

\begin{Lemma}\label{gen_conf}
    If $\nu$ is a $(t,\beta^t)$-conformal measure for $F^j$ with $t>1$, then there exists $M>0$ such that \[
    \liminf_{n\to\infty}\Real(F^{nj}(x))\leq M,
    \]
    for $\nu$-a.e. $x\in\cJ(F)$. In particular, $\nu(I_\infty(F))=0$ or equivalently $\nu(\cJ_r(F))=1$. 
\end{Lemma}
\begin{proof}
For every $k\geq0$, we define inductively functions $\nu_k$ as follows,
\[
\nu_0(A)=\nu(A),\qquad\text{and}\qquad \nu_{k+1}(A)=\int_{F^k(A)}\beta^{-1}\left|\left(\left(F^k|_{F^k(A)}\right)^{-1}\right)'\right|^td\nu_k,
\]
defined over every Borel set $A\subset\cJ(F)$ on which $F^k$ is injective. Using measure-conformality and the ideas in the proof of Proposition \ref{tight}, it is not difficult to see that, if $B\subset (\cJ(F)\cap E_M^c)$ such that $F^k|_B$ is injective, then 
\[\nu_{k+1}(F^{-1}(B))\leq C\beta^{-1}M^{1-t}\nu_k(B).\]
Then by induction, we have 
\[
\nu_n(F^{-n}(B))\le (C\beta^{-1}M^{1-t})^n\nu(V),
\]
for every $n\geq0$. In particular, since by assumption $\nu_j=\nu$, we get 
\[
\nu(F^{-j}(B))\leq (C\beta^{-1}M^{1-t})^j\nu(B). 
\]
Using induction again, 
\[
\nu(F^{-1}(B)\cap...\cap F^{-n}(B))\leq(C\beta^{-1}M^{1-t})^{jn}\nu(B),
\]
for every $n\geq0$. So, for $M>0$ large enough, 
\[
\nu(\bigcap_{n=0}^{\infty}F^{-n}(E_M^c))=0,
\]
and consequently
\[
\nu(\bigcup_{k=0}^\infty F^{-k}(\bigcap_{n=0}^\infty F^{-n}(E_M^c)))=0.
\]
This concludes the proof by noting that the measured in the last equation coincides with $I_\infty(F)$. 
\end{proof}
From Theorem \ref{lim_meas} with $j=1$, and $\beta=e^{P(t)}$, it follows that Lemma \ref{gen_conf} applies for the limit measure $m_t$, and the next consequences follows analogous to Corollary 4.12 in \cite{KU05}. 

\begin{Corollary}\label{cor_4.12}
    For all $M>0$ sufficiently large and $t>1$, we have 
    \[
    m_t(E_M^c)\leq Ce^{(1-t)M},
    \]
    for some constant $C$. 
\end{Corollary}

\begin{theorem}\label{ergodic}
    For $t>1$, the  probability measure $m_t$ is the unique  $(t, e^{P(t)})- $conformal measure for $F$. Moreover, it is ergodic with respect to each  iterate of $F.$
\end{theorem}
\begin{proof}
    Given $t>1$, fix $j\geq1$, and suppose that $\nu$ is a $(t,\beta^j)$-conformal measure for $F^j$ with $\beta>0$. $\nu(I_\infty(F))=0$ by Lemma \ref{gen_conf}. Given $N>0$, set 
    \[
    \cJ_{r,N}(F) =\{z\in\cJ_r(F):\exists\ y\in\cJ(F),\text{ and }\{n_k\}_{k\geq0},\text{ with }y=\lim_k F^{n_k}(z)\text{ and }\Real(y)<N\}.
    \]
    It is easy to see that $\cup_N\cJ_{r,N}(F)=\cJ_r(F)$. Applying Lemma \ref{gen_conf} to $m_t$, there exists $M>0$ such that 
    \[
    \nu(\cJ_{r,N}(F)) = m_t(\cJ_{r,N}(F)) = 1. 
    \]
    Take $z\in\cJ_{r,N}(F)$, and let $0<\delta\leq\dist(\cJ(F),\PC(F))/2$. By definition, there exists $y\in\cJ(F)$ and an increasing sequence $\{n_k\}_k$ such that $y = \lim_k F^{jn_k}(z)$, with $\Real(y)<N$. For $k>1$ large enough, let $F_z^{-jn_k}:B(y,2\delta)\to F_z^{-jn_k}(B(y,2\delta))$ be the holomorphic inverse branch of $F^{jn_k}$ sending $F^{jn_k}(z)$ to $z$, and consider the sets $F_z^{-jn_k}(B(y,2\delta))$ and $F_z^{-jn_k}((B,\delta/K))$. 
    
    Let $\eta = c|(F^{jn_k})'(z)|^{-1}$, $c\asymp 1$. Using $\nu$-conformality, we have 
    \[
    \nu(B(z, \eta)) = \int_{F^{-jn_k}(B(z,\eta))}\beta^{-jn_k}|(F^{-jn_k})'|^td\nu. 
    \]
    By Koebe's distortion theorem, we deduce that 
    \begin{equation}\label{Koebe1}
    B_N(\nu)^{-1}\beta^{-jn_k}|(F^{jn_k})'(z)|^{-t}\leq \nu(B(z,\eta))\leq B_N(\nu)\beta^{-jn_k}|(F^{jn_k})'(z)|^{-t},
    \end{equation}
    where $B_N(\nu)\geq1$ is some constant depending on $\nu$, $N$, and the Koebe's distortion theorem. Let $M$ be fixed as above, take $E$ an arbitrary bounded Borel set contained in $\cJ_r(F)$, and let $E'=E\cap\cJ_{r,M}(F)$. From regularity of $m_t$, for every $x\in E'$, there exists a radius $r(x)>0$ of the form of $\eta$ above with its corresponding  number $n(x)=n_k(x)$ for an appropriate $k$, such that 
    \begin{equation}\label{regular}
    m_t\left(\bigcup_{x\in E'}B(x,r(x))\backslash E'\right)\leq \varepsilon. 
    \end{equation}
    Now, Besicovic theorem allows us to choose a countable sub-cover $\{B(x_i,r(x_i))\}_{i=1}^\infty$ from the cover in equation (\ref{regular}), with $f(x_i)\leq\varepsilon$ and $jn(x_i)\geq\varepsilon^{-1}$, of multiplicity bounded by some constant $C\geq1$, independent of the cover. 

    Suppose $e^{P(t)}<\beta$, from Equations (\ref{Koebe1}) and (\ref{regular}), we have 
    \begin{align}\label{eps_nu}
        \nu(E) = \nu(E') &\leq\sum_{i=1}^\infty \nu(B(x_i,r(x_i)))\beta^{-jn(x_i)} \nonumber \\
        & \leq B_M(\nu)\sum_{i=1}^\infty r(x_i)^t\beta^{-jn(x_i)} \nonumber \\
    & \leq B_M(\nu)B_M(m_t)\sum_{i=1}^\infty m_t(B(x_i,r(x_i)))\beta^{jn(x_i)}e^{P(t)jn(x_i)} \nonumber\\
    & \leq B_M(\nu) B_M(m_t) Cm_t(\bigcup_{i=1}^\infty B(x_i,r(x_i)))\left(e^{P(t)}\beta^{-1}\right)^{jn(x_i)} \\
    & \leq B_M(\nu) B_M(m_t) Cm_t(\bigcup_{i=1}^\infty B(x_i,r(x_i)))\left(e^{P(t)}\beta^{-1}\right)^{\varepsilon^{-1}} \nonumber \\
    & \leq C B_M(\nu) B_M(m_t)\left(e^{P(t)}\beta^{-1}\right)^{\varepsilon^{-1}}(\varepsilon + m_t(E')) \nonumber \\
    & = CB_M(\nu)B_M(m_t)\left(e^{P(t)}\beta^{-1}\right)^{\varepsilon^{-1}}(\varepsilon + m_t(E)). \nonumber
    \end{align}
    If we make $\varepsilon\searrow 0$, given that $e^{P(t)}\beta<1$, we obtain $\nu(E)=0$, since $E$ is arbitrary, it follows that $\cJ(F)=0$, which is a contradiction. We obtain a similar contradiction assuming that $\beta<e^{P(t)}$ and replacing in (\ref{eps_nu}) the roles of $m$ and $\nu$. This way $\beta=e^{P(t)}$ and letting $\varepsilon\searrow0$ again, and exchanging $m_t$ and $v$, we have 
    \[
    \nu(E)\leq CB_M(\nu)B_M(m_t)m_t(E)\qquad\text{and}\qquad m_t(E)\leq CB_M(\nu)B_M(m_t)\nu(E),
    \]
    along with the fact that $m_t(\cJ(F))=\nu(\cJ(F))=1$, we conclude that the measures $m_t$ and $\nu$ are equivalent with Radon-Nikodym derivatives bounded away from zero and infinity. 

    Finally, we prove that any $(t,e^{P(t)})$-conformal measure $\nu$ is ergodic with respect to $F^j$. Suppose $F(G)^{-1}=G$ for some Borel set $G\subset\cJ(F)$ with $0<\nu(G)<1$. Consider the two conditional measures $\nu_R$ and $\nu_{\cJ(F)\backslash G}$, 
    \[
    \nu_G(B) = \dfrac{\nu(B\cap G)}{\nu(G)}, \qquad\text{and}\qquad \nu_{\cJ(F)\backslash G}(B)=\dfrac{\nu(B\cap \cJ(F)\backslash G)}{\nu(\cJ(F)\backslash G)}.
    \]
    It follows that both measures are $(t,e^{P(F)})$-conformal for $F^j$ and mutually singular, which is a contradiction. 
\end{proof}

\section{Bowen Formula}
We start this section, by noticing that all the machinery developed in Section 5 of \cite{KU05} can be applied to our setting. So, it is possible to obtain a fixed point of the normalized operator $\hat{\cL}_t$, restricted to a ``smaller'' Banach space. 

Recall that $C_b:=C_b(\cJ(F))$ is the space of all bounded continuous complex valued functions defined on $\cJ(F)$. For a fixed $\alpha\in(0,1]$, and $g\in C_b$ we define the $\alpha$-variation of $g$ given by 
\[
v_\alpha(g) = \inf\{L\geq0:|g(y)-g(x)|\leq L|y-x|^\alpha\text{ for all }x,y\in\cJ(F)\text{ with}|y-x|\leq \delta\},
\]
where $\delta=\min\{1/2,\dist(\cJ(F),\PC(F)\}$, and let 
\[
\|g\|_\alpha = v_\alpha(g) + \|g\|_\infty. 
\]
It follows that the normed space
\[
(\HH_\alpha,\|\cdot\|_\alpha):= \{g\in C_b:\|g\|_\alpha<+\infty\}
\]
is a Banach space, densely contained in $\cC_b$ with respect to the $\|\cdot\|_\infty$ norm. 

With the observation that for all $t\in\C$, with $\Real(t)\geq0$, all $n\geq1$, all $x,y\in\cJ(F)$ with $|y-x|\leq\delta$, all $v\in F^{-n}(x)$, and some constant $M_t>0$, we have 
\[
\left||(F_v^{-n})'(y)|^t-|(F_v^{-n})'(x)|^t\right|\leq M_t|(F_v^{-n})'(x)|^{\Real t}|x-y|,
\]
Proposition \ref{operator} allows us to say that $\phi_t(z)=e^{-P(t)}|F'(z)|^{-t}$ is a rapidly decreasing summable dynamically Hölder potential (see \cite[Section 5]{KU05} for precise definitions). This way, the normalized operator (\ref{transfer operator}) acts on the Banach space $\HH_\alpha$ with a fixed point. 

\begin{theorem}[Theorem 5.4, \cite{KU05}]
    If $t>1$, then we have the following. 
    \begin{itemize}
        \item [(a)] The number 1 is a simple isolated eigenvalue of the operator $\hat{\cL}_t:\HH_\alpha\to\HH_\alpha$. 
        \item[(b)] The eigenspace of the eigenvalue 1 is generated by nowhere vanishing function $\psi_t\in\HH_\alpha$ such that $\int\psi_tdm_t=1$ and $\lim_{\Real(z)\to\infty}\psi_t(z)=0$. 
        \item[(c)] The number 1 is the only eigenvalue of modulus 1.  
    \end{itemize}
\end{theorem}
We give a short sketch of the proof based on \cite{KU05}.
\begin{proof}[Sketch]
    Combining Lemmas 5.1 and 5.2 in \cite{KU05}, there exists a sequence of positive integers $\{n_k\}_{k\geq1}$ such that 
    \begin{equation}\label{psi}
    \dfrac{1}{n_k}\sum_{j=1}^{n_k}\hat{\cL}_t^j(\mathds{1})\to \psi_t,
    \end{equation}
    as $k\to\infty$, for some $\psi_t:\cJ(F)\to \R$ in the Banach space $\HH_\alpha$. By construction, $m_t$ is a fixed point of the operator conjugate to $\hat{\cL}_t$, which implies that 
    \[
1=\int \dfrac{1}{n}\sum_{j=0}^{n-1}\hat{\cL}_t^j(\mathds 1)dm_t\to \int\psi_tdm_t,\ \text{as }n\to\infty.
    \]
    Applying $\hat{\cL}_t$ to the limit in (\ref{psi}) we have that $\psi_t=\hat{\cL}_t\psi_t$, and Proposition \ref{operator} implies that $\lim_{z\to+\infty}\psi_t(z)=0$. 

    Analogous to the proof in \cite{KU05}, by the ergodicity of each iterate of $F$ (Theorem \ref{ergodic}), we have that if $\beta\in S^1$ is an eigenvalue of $\hat{\cL}_t$ and $\rho$ is its eigenfunction, then $\beta=1$ and $\rho\in\C\psi_t$.
    
\end{proof}

We are now in position to state finer stochastic properties of an $F$-invariant measure obtained from Theorem \ref{psi}. The following result sum-up the theorems \cite[Section 6]{KU05}, we refer to this section for more details.

\begin{theorem}
    For $t>1$, defined the measure $\mu:=\mu_t=\psi_tm_t$. Then 
    \begin{itemize}
        \item [(1)] $\mu$ is $F$-invariant, ergodic with respect to each iterate of $F$ and equivalent to the measure $m_t$. In particular, $\mu(\cJ_r(F))=1$.
        \item[(2)] The dynamical system $(F,\mu_t)$ is metrically exact. 
        \item[(3)] If $g\in\HH_\alpha$, $\alpha\in(0,1)$ are such that the \emph{asymptotic variance (or dispersion)} $\sigma^2(g)$
        \[
        \sigma^2(g) = \lim_{n\to\infty}\dfrac{1}{n}\int\left(\sum_{j=0}^{n-1}g\circ F^j-nEg\right)^2d\mu_t
        \]
        exists, then the sequence of random variables $\{g\circ F^n\}_{n=0}^{\infty}$ with respect to the probability measure $\mu_t$ satisfies the Central Limit Theorem. 
    \end{itemize}
\end{theorem}
As was mentioned by Kotus and Urba\'nski in \cite{KU05}, $\cJ_r(F)$ is the right set to study. Set 
\[
\cJ_{bd}(F)=\{z\in\cJ(F):\inf\{\Real(F^n(z)):n\geq0\}<+\infty\},
\]
it follows that $\cJ_{bd}(F)\subset \cJ_r(F)$. We are now in position to prove the following. 
\begin{theorem}\label{haus_bounded}
    We have $\HD(\cJ_{bd}(F))>1$. 
\end{theorem}

\begin{proof}
    The idea to prove this result, is to construct an iterated functions systems $H_R$, and obtain a lower bound for the Hausdorff dimension of their limit sets $J_R$. In this direction, consider the following set, 
    \[
    S_R = \{z\in P: R\leq \Real(z)\leq 4R, \text{ and } 0<\varepsilon_R<\Ima(z)<2\pi-\varepsilon_R\},
    \]
    where the value of $\varepsilon_R>0$ will be given through the proof. 

    Now, fix an integer $k\geq1$, and consider the inverse map $F_k^{-1}:S_R\to P$ given by the formula 
    \[
    F_k^{-1}(z) = f_*^{-1}(z+2k\pi i). 
    \]
    Recall $f_*^{-1}:\C\to P$ is the inverse of the bijection $f:P\to \C$ and $\lambda=c-(\ell-1)\log c$. To obtain the IFS on $S_R$, we need values $k\geq1$ and $R\geq1$ such that $F_k^{-1}(S_R)\subset S_R$. First, note that 
    \[
    z+2k\pi i = \ell F_k^{-1}(z) + \lambda -e^{F_k^{-1}(z)},
    \]
    hence 
    \[
    |z+2k\pi i|\leq |e^{F_k^{-1}(z)}|+\ell|F_k^{-1}(z)|+|\lambda|\leq 2e^{\Real(F_k^{-1}(z))}, 
    \]
    for every $k$ sufficiently large. Also, in the other direction, 
    \[
    |z+2k\pi i| \geq |e^{F_k^{-1}(z)}| - \ell|F_k^{-1}(z)|-|\lambda|\geq \dfrac{1}{2}e^{\Real(F_k^{-1}(z))}, 
    \]
    so, if $k\in[e^{2R},e^{3R}]$, with $R>>1$, then 
    \[
    R\leq \Real(F_k^{-1}(z))\leq 4R.
    \]
    Now, if we take $w=x+iy\in P\setminus (S_R\cap \{R\leq \Real(z)\leq 4R\})$, we have 
    \begin{align*}
        |\Ima(f(w))| & = |\ell\Ima(w) + \Ima(\lambda) - \Ima(e^w)| \\ 
            & \leq e^x|\sin y| + 2|y| + |\Ima(\lambda)| \\
            & \leq e^{4R}\sin (\varepsilon_R) + 4\pi + |\Ima(\lambda)| \\
            & \leq \dfrac{1}{2}e^R,
    \end{align*}
    for $\varepsilon_R>0$ sufficiently small. Hence, taking such $\varepsilon_R>0$ and $k\in[e^R,e^{4R}]$, we have the desired inclusion 
    \[
    F_k^{-1}(S_R)\subset S_R. 
    \]
    This way, for each $k\in [e^R,e^{4R}]$, $R>1$ sufficiently large, we have constructed an IFS over $S_R$ given by 
    \[
    H_R:= \{F_k^{-1}:S_R\to S_R\}. 
    \]
    Following the ideas of Mauldin \& Urba\'nski for IFS's, it is not difficult to see that each IFS $H_R$ satisfies the requirements and hence the results in \cite{MaUr}. Let $J_R$ the limit set of the system $H_R$. By construction, $J_R\subset \cJ_{bd}(F)$, so it is enough to show that $\HD(J_R)>1$. To estimate this, note that if $w=f(w) = \ell z + \lambda - e^z$, then 
    \begin{equation}\label{deri_inv}
    (f_*^{-1})'(w) = \dfrac{1}{w - \ell f_*^{-1}(w) - \lambda + \ell}. 
    \end{equation}
    Let $z\in S_R$, and apply (\ref{deri_inv}), then 
    \begin{align}\label{HDHR}
        |(F_k^{-1})'(z)| & = \dfrac{1}{|(z+2k\pi i) - \ell f_*^{-1}(z+2k\pi i) + \ell -\lambda|}\nonumber \\
            & \geq \dfrac{1}{|z+2k\pi i| + \ell|f_*^{-1}(z+2k\pi i)| + \ell + |\lambda|}\nonumber \\
            & \geq \dfrac{1}{|z+2k\pi i| + 8R + 2\ell\pi + |\lambda|} \\
            & \geq \dfrac{1}{9k + 4\ell R + 2\ell \pi +|\lambda|}\nonumber\\
            & \geq \dfrac{1}{10k},\nonumber
    \end{align}
    taking $R>1$ sufficiently large and $k\in [e^R,e^{4R}]$. Let $P_R(t)$ be the topological pressure of the IFS $H_R$ evaluated at $t>0$. Then by (\ref{HDHR})
    \begin{align*}
        P_R(1) & = \log \left(\sum_{k=e^{2R}}^{e^{3R}}|(F_k^{-1})'(z)|\right)   \geq \log \left(\sum_{e^{2R}}^{e^{3R}}\dfrac{1}{10k}\right) 
          = -\log 10 + R > 0,
    \end{align*}
    which implies that $\HD(J_R)>1$ and then $\HD(\cJ_{bd}(F)>1$. 
\end{proof}

When the base space (in this case, the Julia set $\cJ(f)$) is compact, the common value of the limits $\overline{P}(t,F)$ and $\underline{P}(t,F)$, denoted simply by $P(t)$, obtained in Proposition \ref{operator}, corresponds to the well-known \emph{geometric pressure function} for $t>1$. Moreover, when this is the case, it is possible to get a link between the ergodic theory and the geometry of the embedding of the Julia set in the Euclidean space $\R^2$. This link is among the concepts of \emph{pressure function, entropy, Hausdorff and Package dimension, and the Lyapunov characteristic exponent}. The entropy and the Lyapunov exponent lose meaning when the map $F$ is transcendental or the Julia set is non-compact. 

Nevertheless, the properties for $P(t)$ in Proposition \ref{operator} and the ones listed below, allow us to call this function, the corresponding geometric pressure function of exponent $t>1$. 

\begin{proposition}\label{zero_press}
    The function $t\to P(t),$ $t\geq 0$ satisfies
    \begin{enumerate}
        \item There exists $t\in (0,1)$ such that $0\leq P(t)<+\infty$
        \item There exists exactly one $t>1$ such that $P(t)=0.$
        
    \end{enumerate}
\end{proposition}
\begin{proof}
    Following \cite{KU05}, we consider the iterated function system $H_R$ defined in the above theorem. For $R>1$ large enough $\HD(J_R)>1$. We have $P(\HD(J_R))\geq P_R(\HD(J_R))=0$, combining this with (4) of Proposition \ref{operator} we conclude (1). The item (2) is a consequence of the remaining properties of $P(t)$. 
\end{proof}

On the other hand, since $\cJ(F)$ is not a compact set (in the complex plane $\C$), we cannot apply the Ergodic Birkhoff Theorem to obtain the Lyapunov characteristic exponent as the limit of the Birkhoff sums applied to a point of $\cJ(F)$.  Even so, using the expanding character of $F$ when $z\to -\infty$, and Corollary \ref{cor_4.12} we can define this exponent directly as an integral. 

We set the following auxiliary subsets of each $E_n$, $n>0$. 
\[
A_n = \{z\in\cJ(F):n\leq \Real z\leq n+1\}. 
\]

\begin{Lemma}\label{lyap}
    For $t>1$, $\int\log |F'|d\mu_t<+\infty$, i.e., the Lyapunov exponent of $F$, $\chi(F)$ with respect to $\mu_t$, is well-defined.
\end{Lemma}
\begin{proof}
    Given that 
    \[
    \lim_{z\to+\infty}\psi_t(z)=0, 
    \]
    it follows that $\psi_t:\cJ(F)\to(0,\infty)$ is bounded. From the form of the derivative of $f$, (\ref{deri}), and Corollary \ref{cor_4.12}, we have 
    \begin{align*}
        \int\log|F'|d\mu_t & \preceq \int\log|F'|dm_t \\
        & = \sum_{n=1}^\infty \int_{A_n}\log|F'|dm_t \\
        & \leq \sum_{n=1}^\infty m_t(E_n^c)\log (1+e^n) \\
        &\leq \sum_{n=1}^\infty C\exp \left(1-tn\right)(1+n)<+\infty.
    \end{align*}
    We have then a well-defined Lyapunov exponent $\chi(F)=\int\log|F'|d\mu_t<\infty$ with respect to $\mu_t$. 
\end{proof}

Bowen's formula is one of the most significant results in thermodynamic formalism, it relates the Hausdorff dimension of a hyperbolic invariant repeller with the corresponding pressure function. In this case, although the pressure function $P(t)$ is defined on the whole Julia set, and not just the residual Julia set $\cJ_r$, the following version of Bowen's formula is a justification of why to study the residual Julia set. For completeness and given the nature of the next result we add its proof following the arguments in \cite{KU05}.

\begin{theorem}[Bowen's Formula]\label{bowen}
    The Hausdorff dimension of the residual Julia set $\cJ_r(F)$ is the unique zero of the pressure function $t\mapsto P(t)$, for $t>1$. 
\end{theorem}
\begin{proof}
    Let $\eta>1$ be the unique zero of the pressure function. Note that $\eta>1$ is well-defined from the general properties of the pressure function $P(t)$ as presented in Proposition~\ref{zero_press}.
    For $k\geq1$, let 
    \[
    X_k=\{z\in\cJ(F):\liminf_{n\to\infty} \Real (F^n(z)) <k\}. 
    \]
    And let $z\in\cJ(F)$ be an arbitrary point. Fix $t>\eta$, and take $n\geq1$ large enough such that 
    \[
    \dfrac{1}{j} \log \sum_{x\in F^{-j}(z)}|(F^j)'(x)|^{-t}\leq \dfrac{1}{2} P(t), \forall j\geq n,
    \]
    and note that $P(t)<0$ since $P(t)$ is strictly decreasing. Consider a finite cover of $E_k$ by open balls $B(z_1,\delta)$, $B(z_2,\delta)$, ...,$B(z_l,\delta)$. It is not difficult to see that 
    \[
    X_k\subset \bigcup_{j=n}^\infty \bigcup_{x\in F^{-j}(z_j)}F_x^{-j}(E_k),
    \]
    where $\delta>0$ is such that $F_x^{-j}$ are the corresponding local inverse branch, as previously defined. If $\HH^t(\cdot)$ denotes the $t$-dimensional Hausdorff measure, we conclude that 
    \[
    \HH^t(X_k)\leq \lim_{n\to\infty}\sum_{j=n}^\infty \sum_{i=1}^l \sum_{x\in F^{-j}(z_i)}K^t|(F^j)'(x)|^{-t}(2\delta K)^t\leq l(2\delta K)^t\lim_{n\to\infty}\sum_{j=n}^\infty \exp\left(\dfrac{1}{2}P(t)j\right) = 0,
    \]
    since $P(t)<0$, $t>\eta$. It follows that $\HD(X_k)\leq t$. But $\cJ_r(F)=\cup_{k=0}^\infty X_k$, we have $\HD(\cJ_r(F))\leq t$. Letting $t\searrow\eta$ gives the inequality $\HD(\cJ_r(F))\leq\eta$. 

    Now to prove the opposite inequality, we fix $\varepsilon>0$. We know that $\mu_\eta(\cJ_r(F))=1$ and that $\mu_\eta$ is ergodic and $F$-invariant, so from the Birkhoff's ergodic theorem and Egorov's theorem, there exist a Borel set $Y\subset\cJ_r(F)$ and an integer $k\geq1$ such that $\mu_\eta(Y)\geq\frac{1}{2}$ and for every $x\in Y$, and every $n\geq k$, 
    \begin{equation}\label{lim_lyap}
    \left|\dfrac{1}{n}\log|(F^n)'(x)|-\chi(F)\right|<\varepsilon, 
    \end{equation}
    where $\chi(F)=\inf\log|F'|d\mu_\eta$ is the finite Lyapunov exponent due to Lemma \ref{lyap}. Put $\nu=m_{\eta|Y}$. Given that $x\in Y$, and $0<r<\delta$, let $n\geq0$ be the largest integer such that 
    \begin{equation}\label{contention7.2}
        B(x,r)\subset F_x^{-n}(B(F^n(x),\delta)). 
    \end{equation}
    Then $B(x,r)$ is not contained in $F_x^{-(n+1)}(B(F^{n+1}(x),\delta))$, applying $\dfrac{1}{4}$-Koebe's distortion theorem, we get 
    \begin{equation}\label{r-1/4Koebe}
        4r\geq \delta|(F^{n+1})'(x)|^{-1}. 
    \end{equation}
    By taking $r>0$ sufficiently small, we may assume that $n\geq k$. Combining (\ref{contention7.2}), $P(\eta)=0$ applied to the conformality of $m_\eta$, Koebe's Distortion Theorem and (\ref{r-1/4Koebe}), we obtain
    \begin{align*}
        m_\eta(B(x,r) & \leq m_\eta(F_x^{-n}(B(F^n(x),\delta))) \\
        & \asymp |(F_x^{-n})'(F^n(x))|^{\eta}m_\eta(B(F^n(x),\delta)) \\ 
        & \preceq |(F_x^{-n})'(F^n(x))|^\eta \\
        & \preceq r^\eta \dfrac{|(F^{n+1})'(x)|^\eta}{|(F^n)'(x)|^\eta}. 
    \end{align*}
    From (\ref{lim_lyap}), we deduce that 
    \begin{equation}\label{eq7.4}
        m_\eta(B(x,r))\leq r^\eta e^{(\chi+\varepsilon)(n+1)}e^{-(\chi-\varepsilon)n}\asymp r^\eta e^{2\varepsilon n}. 
    \end{equation}
    Combining (\ref{contention7.2}), Koebe's distortion theorem and (\ref{lim_lyap}), we obtain 
    \[
    r\leq K\delta|(F^n)'(x)|^{-1}\leq K\delta e^{-(\chi-\varepsilon)n}.
    \]
    This way, $e^{(\chi-\varepsilon)n}\preceq r^{-1}$, hence $e^{2\varepsilon n}\preceq r^{-\frac{2\varepsilon}{\chi-\varepsilon}}$, which combined with (\ref{eq7.4}) imply that 
    \[
    \nu(B(x,r)) \leq m_\eta(x,r))\preceq r^{\eta - \frac{2\varepsilon}{\chi-\varepsilon}}.
    \]
    So, $\HD(\cJ_r(F))\geq\HD(\nu)\geq\eta- \frac{2\varepsilon}{\chi-\varepsilon}$. Finally, if we make $\varepsilon\to0$, we have that $\HD(\cJ_r(F)\geq\eta$, which concludes the proof.
\end{proof}

The proof of the following results follows verbatim (in our context) as in \cite[Proposition 7.4]{KU05}.
\begin{proposition}\label{prop7.4}
    Let $P^h(\cdot)$ denotes the $h$-dimensional packing measure. We have $P^h(\cJ_r(F))=\infty$. In fact, $P^h(G)=\infty$ for every open non-empty subset of $\cJ_r(F)$. 
\end{proposition}
Combining Theorem \ref{haus_bounded} and Proposition \ref{prop7.4} we have. 

\begin{Corollary}
    It holds $1<\HD(\cJ_r(F))<2$.
\end{Corollary}

\section{Graphical examples}
In this section, we present a few examples from the family $f_{\ell,c}$ for different values of $\ell$ and corresponding parameters $c$ as an illustration of the complexity of the Julia sets, the topology of the invariant Baker domain, the immediate basin of attraction, and the pair of sets of wandering domains. 

\subsection{The one from Bergweiler} The first example is the function studied by Bergweiler in \cite{Ber95}. For $\ell=c=2$ we have 
\[
f_{2,2}(z) = 2 -\log 2 + 2z - e^z. 
\]
For this function, the fixed point $z_c=\log 2$ is also a critical value, in other words, is a \emph{super-attracting} fixed point. This way, the post-singular set not only coincides with the singular value, but attracts all the dynamics in the immediate basin of attraction $U_0$ and all of its vertical translates $U_n=U_0+2\pi ni$
\begin{figure}[ht]
\includegraphics[width=14cm]{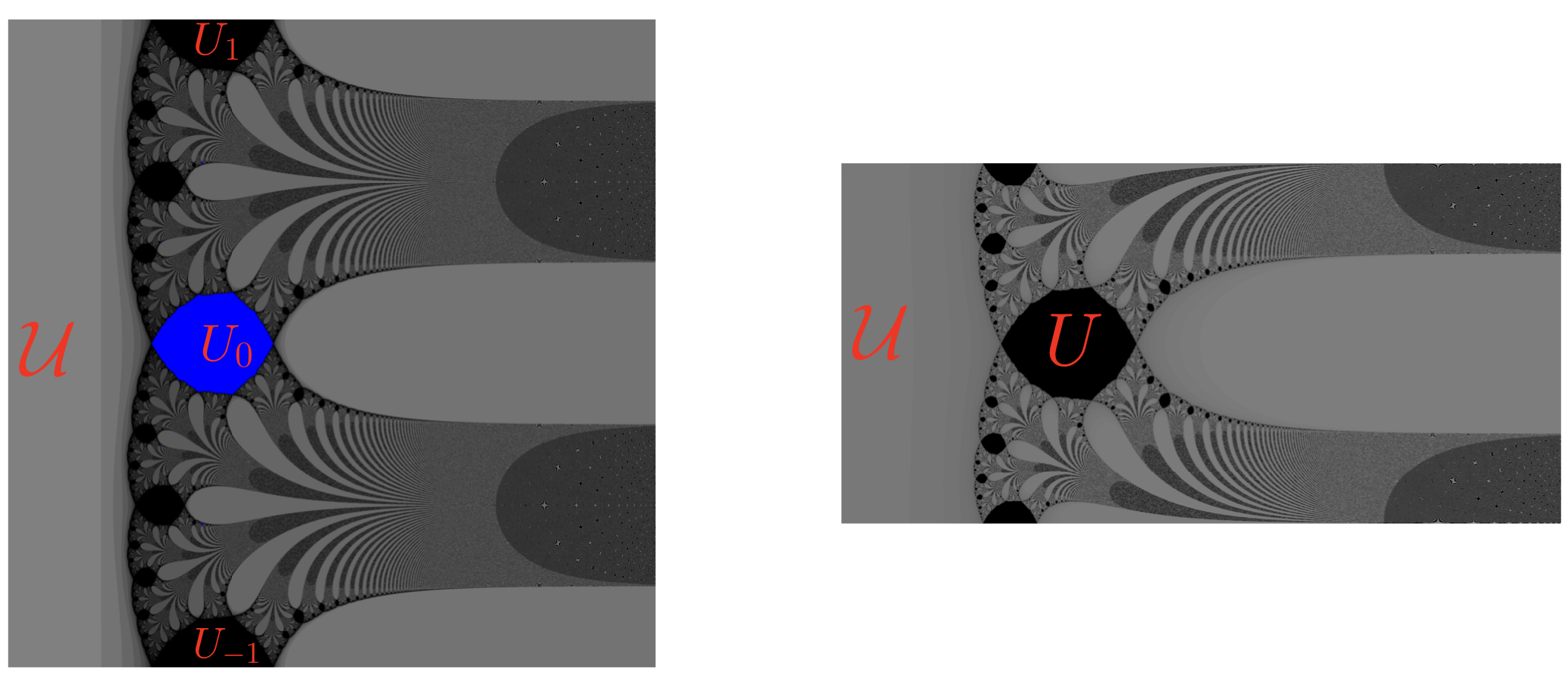}
\centering
\caption{LEFT: Dynamical plane of the function $f_{2,2}$. $U_0$ is the attracting domain (blue), $\mathcal{U}$ is the Baker domain (gray-tones), while $U_i$ represents the wandering domains (black). RIGHT: The simulated dynamics in the quotient space. This case there is no wandering domains.}
\end{figure}
\subsection{A general case} For this example we consider a bigger post-singular set. Take $\ell=3$ and $c=2.7-0.3i$, thus
\[
f_{3,c}(z) = 2.7-0.3i -2\log (2.7-0.3i) + 3z - e^z. 
\]
\begin{figure}[ht]
\includegraphics[width=10cm]{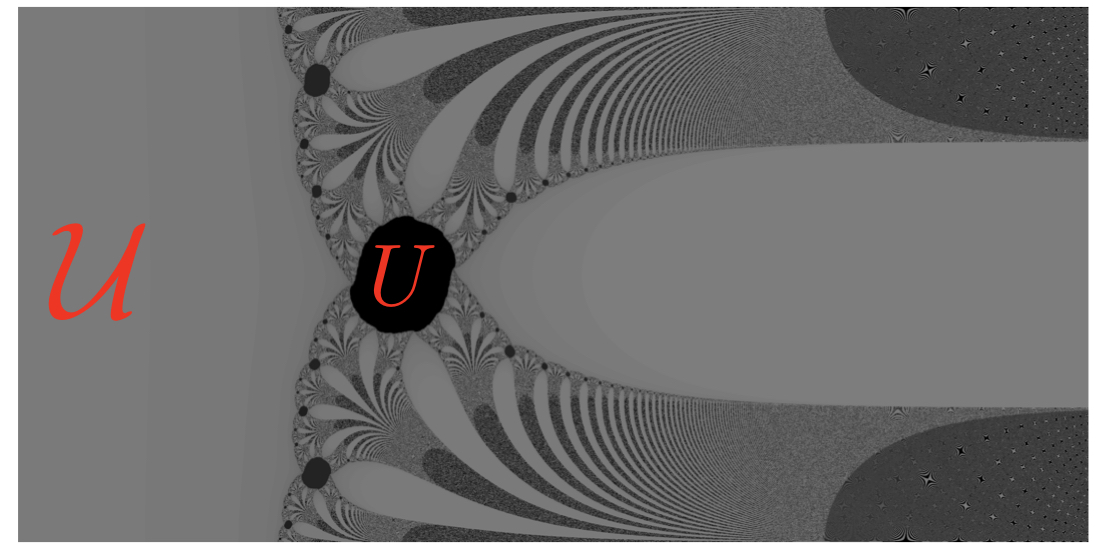}
\centering
\caption{The simulated dynamics of $F_{3,c}$ in the quotient space. Again, there are no wandering domains. }
\end{figure}
Since the critical point $z_0=\log 3$ is no longer a fixed point, the post-singular set is (much) bigger than the singular one. Here, the fixed point $z_c=\log (2.7-0.3i)$ is only \emph{attracting}, but its immediate basin contains the critical point $z_0=\log 3$, so the set $\{z_c+2k\pi i:k\in\Z\}$ coincides with the derived set of the post-singular set. 
\subsection{The one from Fatou family}
As is mentioned in the Introduction, the family $\flc$ contains more dynamics than the Fatou family $f_\lambda(z)=z+\lambda-e^z$. The restriction on $\ell\geq2$, allows the existence of the Baker domain for every $c\in D(\ell,1)$. If we consider $\ell$, and $c\in D(1,1)$, then, the Fatou set has no Baker domain, actually the Fatou set consists of the basin of attraction of the point $z_c=\log c$ and the wandering sets of its vertical translates. In order to $\flc$ belongs to the Fatou family, we need that $\Real(c)<0$. Then, the left half-plane $\Real(z)<0$ is contained in the invariant Baker domain, which is the only component of the Fatou set. For this example we take $\ell=1$ and $c=-0.5$.
\[
f_{1,-0.5}(z) =  z -0.5 - e^z. 
\]
\begin{figure}[ht]
\includegraphics[width=12cm]{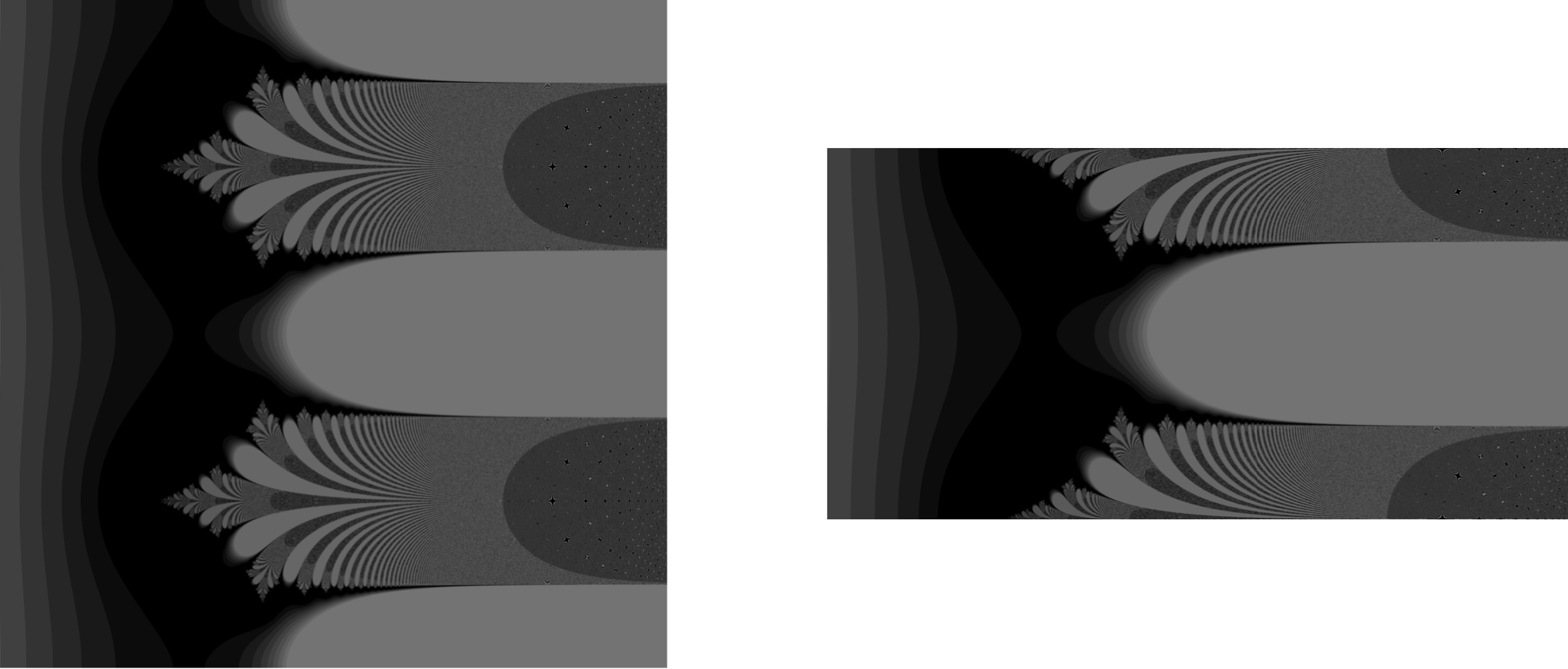}
\centering
\caption{LEFT: Dynamical plane of the function $f_{1,-1/2}$. In this case, there is only one invariant Baker domain (non-univalent). RIGHT: The simulated dynamics in the quotient space. The Fatou set consists of the only Baker domain.}
\end{figure}
\subsection{A boundary parameter case} 
In our setting, the Fatou family is denoted by the form 
\[
f_\lambda(z) = z + \lambda - e^z,\ \Real(\lambda)<0.
\]
And for $\ell=1$, $c\in D(1,1)$, the family $f_{1,c}$ is reduced to 
\[
f_{1,c}(z) = z +c - e^z. 
\]
It is not difficult to notice that there is one, and only one, common boundary point in the two families, corresponding to the parameter $c=0$, obtaining the function
\[
f(z) = z - e^z. 
\]
\begin{figure}[ht]
\includegraphics[width=14cm]{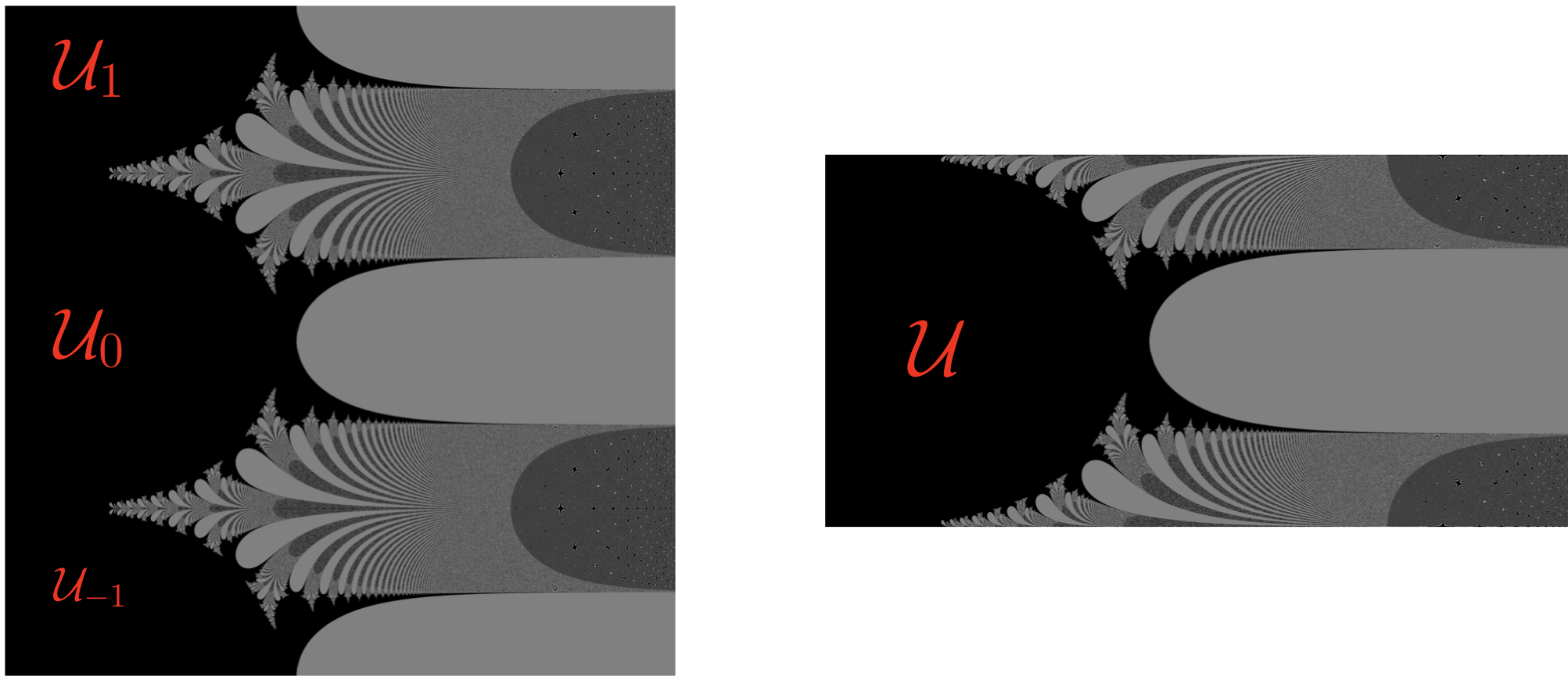}
\centering
\caption{LEFT: Dynamical plane of the function $f$. Each $U_i$ is an invariant Baker domain, each contained in a strip of height $2\pi$. RIGHT: The simulated dynamics in the quotient space. There is only one Baker domain}
\end{figure}
This function has been widely studied in different aspect, from the geometric, to the measurable point of view. The function has been studied in the form $z\mapsto z+e^{-z}$, which is a conformal conjugation under the map $z\mapsto -z$. Some of the important features this function posses are the following. 

In \cite{fh06}, the authors proved that this is a rigid map, i.e., with a trivial deformation space. In \cite{ds15}, it was shown that this function can be obtained as a pinching process from the Fatou function $f_1(z)=z-1-e^z$. Unlike the functions in the Fatou family, $f_{1,0}$ possesses countable many Baker domains, each one containing a unique singular value. 

In terms of measure-theoretic aspects, it is well-known that the dynamics on the boundary of the Baker domains is ergodic. Moreover, the escaping set $\cI(f)$ has zero harmonic measure. Recently in \cite{fj23} the authors proved stronger properties of the boundary set. Results in \cite{fj23} describe dynamically, topologically, and combinatorial the escaping points in the boundary of the Fatou set. 

Since $f_{1,0}$ belongs to the class defined by Stallard in \cite{Sta90}, and respects the equivalence relation, it is natural to ask if the above machinery can be applied to this particular function.

\end{document}